%% file: main.tex
\documentclass[10pt,reqno]{amsart}

\usepackage{graphicx}
\usepackage{booktabs}
\usepackage{hyperref}
\usepackage{ifthen}
\usepackage{mathptmx}
\usepackage{amsmath}
\usepackage{verbatim} 
\usepackage{amscd}
\usepackage{amssymb}
\usepackage{amsthm}
\usepackage{xspace}
\usepackage{amsmath, amsthm, amssymb,color}
\usepackage{enumitem}
\usepackage{bm}
\newcommand{\showcomments}{yes}

\newsavebox{\commentbox}
{\ifthenelse{\equal{\showcomments}{yes}}%
{\footnotemark
        \begin{lrbox}{\commentbox}
        \begin{minipage}[t]{1.25in}\raggedright\sffamily\tiny
        \footnotemark[\arabic{footnote}]}
{\begin{lrbox}{\commentbox}}}%
{\ifthenelse{\equal{\showcomments}{yes}}%
{\end{minipage}\end{lrbox}\marginpar{\usebox{\commentbox}}}
{\end{lrbox}}}

\newcommand{\compat}[1]{{\ifthenelse{\equal{\showcomments}{yes}}{\textcolor{red}{Pat: #1}}{}}}
\newcommand{\comkha}[1]{{\ifthenelse{\equal{\showcomments}{yes}}{\textcolor{blue}{Kha: #1}}{}}}

\theoremstyle{plain}
\newtheorem{thm}{Theorem}[section]
\newtheorem{cor}[thm]{Corollary}
\newtheorem{prop}[thm]{Proposition}
\newtheorem{lemma}[thm]{Lemma}
\newtheorem{ques}{Question}
\newtheorem{conj}{Conjecture}

\theoremstyle{definition}
\newtheorem{rem}[thm]{Remark}

\DeclareMathOperator{\rank}{rank}

\DeclareMathOperator{\Aut}{Aut} \DeclareMathOperator{\Out}{Out}
 
\DeclareMathOperator{\SL}{SL} \DeclareMathOperator{\PSL}{PSL}
\DeclareMathOperator{\GL}{GL}

\DeclareMathOperator{\Aff}{Aff}

\DeclareMathOperator{\Inn}{Inn}

\DeclareMathOperator{\diag}{diag}

\newcommand{\Q}{\ensuremath{\mathbb{Q}}}
\newcommand{\R}{\ensuremath{\mathbb{R}}}
\newcommand{\Z}{\ensuremath{\mathbb{Z}}}

\newcommand{\C}{\ensuremath{\mathbb{C}}}

\newcommand{\T}{\ensuremath{\mathbb{T}}}

\newcommand{\CA}{\mathcal A}

\newcommand{\CI}{\mathcal I}

\newcommand{\CS}{\mathcal S} \newcommand{\CT}{\mathcal T}

\newcommand{\va}{{\mathbf{a}}}
\newcommand{\vb}{{\mathbf{b}}}
\newcommand{\e}{{\mathbf{e}}}

\newcommand{\vv}{{\mathbf{v}}}
\newcommand{\vw}{{\mathbf{w}}}

\newcommand{\Conj}{\mathrm{conj}}
\newcommand{\0}{\boldsymbol{0}}% bold zero
\newcommand{\eval}{\mathrm{eval}}
\newcommand{\bPsi}{\overline{\Psi}}%

\DeclareMathOperator{\img}{img}

\title{The Extrinsic Primitive Torsion Problem}

\author{Khalid Bou-Rabee}
\email{kbourabee@ccny.cuny.edu}
\address{Department of Mathematics \\ 
The City College of New York \\ 
160 Convent Ave \\ 
New York, NY 10031}
\address{
Department of Mathematics \\ 
The Graduate Center, CUNY \\ 
365 Fifth Avenue \\
New York, NY 10016
}

\author{W. Patrick Hooper}
\email{whooper@ccny.cuny.edu}
\address{Department of Mathematics \\ 
The City College of New York \\ 
160 Convent Ave \\ 
New York, NY 10031}
\address{
Department of Mathematics \\ 
The Graduate Center, CUNY \\ 
365 Fifth Avenue \\
New York, NY 10016
}
\begin{document}
\begin{abstract}
Let $P_k$ be the subgroup generated by $k$th powers of primitive elements in $F_r$, the free group of rank $r$.
We show that $F_2/P_k$ is finite if and only if $k$ is $1$, $2$, or $3$.
We also fully characterize $F_2/P_k$ for $k = 2,3,4$.
In particular, we give a faithful nine dimensional representation of $F_2/P_4$ with infinite image.
\end{abstract}
\maketitle

\subsection*{Keywords}
The Burnside Problem, Primitive elements, Characteristic subgroups

%\tableofcontents

\section{Introduction}

% Let $G$ be a group such that the smallest cardinality of a generating set of $G$ is $r$. Recall that $g \in G$ is $r$-\emph{primitive} if it is part of a generating set of $G$ with $r$ elements. 
% The \emph{rank} of a group is the cardinality of a generating set of the group of minimal size.
% Denote the rank $r$ free group by $F_r$. 
% Clearly, the set of primitive elements generates $F_k$ for any $k$. 
% The main goal of this paper is to determine when fixed powers of primitive elements enjoy a similar property.
% We collect this question, along with related ones, into the following.

Let $G$ be a group and $r$ be a cardinality. We say that $g \in G$ is $r$-\emph{primitive} if it is part of a generating set of $G$ with $r$ elements. The \emph{rank} of a group $G$ is the cardinality of a generating set of minimal size, and an element of $G$ is called \emph{primitive} if it is $r$-primitive with $r$ equal to the rank of $G$. Denote the rank $r$ free group by $F_r$. %Clearly, the set of primitive elements generates $F_r$ for any $r$. 
%The main goal of this paper is to determine when fixed powers of primitive elements enjoy a similar property. 
This paper concerns the following collection of questions.

\begin{ques}[The Extrinsic Primitive Torsion Problems]
\label{q:extrinsic}
Fix positive integers $r$ and $k$. Let $\Gamma$ be an image of $F_r$ such that the image of every $r$-primitive element in $F_r$ has order dividing $k$.
\begin{enumerate}[label=\emph{(\alph*)}]
\item Is $\Gamma$ necessarily finite?
\item Is $\Gamma$ necessarily virtually nilpotent?
\item Is $\Gamma$ necessarily virtually solvable?
\item Is $\Gamma$ necessarily finitely presented?
\end{enumerate}
What if $\Gamma$ is as above and also residually finite?
\end{ques}

\noindent
Observe that a positive answer to Questions \ref{q:extrinsic} (a) or (b) implies a positive answer to Question \ref{q:extrinsic} (d).

The Extrinsic Primitive Torsion Problems are topological variants of the classical Burnside Problem posed by William Burnside in 1902 \cite{WB1902}. This problem has led to many important discoveries: the classical Jordan-Schur Theorem, the A. Ju. Ol'\v{s}anski\u{\i}'s outrageous Monster groups \cite{MR571100}, and the fundamental Golod--Shafarevich Theorem \cite{MR0161852}.
As such Question \ref{q:extrinsic} is intrinsically motivated through group theory (moreover, it increases our understanding of new characteristic subgroups of free groups). The case of $r=2$ has direct ties to geometric questions about square-tiled surfaces; please see Appendix A. 

There has been significant progress made on the Primitive Torsion Problem for some sufficiently large $k$ \cite{MR3570293, MP2017}.
This paper answers Question \ref{q:extrinsic} (a) in the case $r = 2$.
This paper also answers Question \ref{q:extrinsic} (a)-(d) in the cases $r=2$ and $k \in \{2,3,4\}$.

%The answers to all questions in these cases are yes, except that that when $r=2$ and $k=4$, the group $\Gamma$ need not be finite.
%For $r = 2$ and $k \geq 5$, we show that the answer to Question \ref{q:extrinsic}\,(a) is no.
%We do not consider the case $r > 2$ in this paper.
%We do not know the answers for parts (b), (c), or (d) for $r = 2$ and $k \geq 5$ and we do not consider the case $r >2$ in this paper. 

We will succinctly state our findings in a table.
Let $P_{r,k} \subset F_r$ be the subgroup generated by $k$th powers of primitive elements in $F_r$ (observe that the answer to Questions \ref{q:extrinsic} (a), (b), or (c) is affirmative if and only if the respective answer to (a), (b), or (c) is affirmative for $\Gamma=F_r/P_{r,k}$).
Use $P_k$ to denote $P_{2,k}$ and use $H(R)$ to denote the Heisenberg group over a ring $R$.
\vspace{1em}

\begin{center}
\begin{tabular}{l l l}
\toprule
Subgroup & Index in $F_2$ &  Quotient $G_k=F_2/P_k$ \\
\midrule
$P_2$ & $4$ & The Klein four-group. \\
$P_3$ & $27$ &  $H(\Z/3)$. \\
$P_4$ & $\infty$ & Virtually a five dimensional image of $H(\Z) \times H(\Z)$. \\
$P_5$ & $\infty$ & \emph{We conjecture virtually solvable.} \\
$P_k$ with $k \geq 6$ & $\infty$ & \emph{We conjecture the quotient is not finitely presented.} \\
\bottomrule
\end{tabular}
\end{center}
\vspace{1em}

In resolving the cases $k = 4$ we show that $F_2/P_4$ is isomorphic to the matrix group generated by the following two matrices:
$$\diag(1, -1, -i, -i; -1, 1, i, i; 1),$$
$$\left(\begin{array}{rrrr|rrrr|r}
0 & 0 & 0 & 0 & 0 & 1 & 1 & 0 & 0 \\
0 & 0 & 0 & 0 & -1 & 0 & 0 & 1 & 0 \\
0 & 0 & 0 & 0 & 0 & 0 & 1 & 0 & 1 \\
0 & 0 & 0 & 0 & 0 & 0 & 0 & 1 & 0 \\
\hline
0 & -1 & -1 & 0 & 0 & 0 & 0 & 0 & 0 \\
1 & 0 & 0 & -1 & 0 & 0 & 0 & 0 & 0 \\
0 & 0 & -1 & 0 & 0 & 0 & 0 & 0 & 0 \\
0 & 0 & 0 & -1 & 0 & 0 & 0 & 0 & -1 \\
\hline
0 & 0 & 0 & 0 & 0 & 0 & 0 & 0 & 1
\end{array}\right).$$

For $k \geq 5$, we develop tools for constructing and refining new infinite linear representations of $F_2/P_k$.
These tools allow us to answer Question \ref{q:extrinsic}\, (a), and we hope they will be useful in future work.

Instead of speaking of primitivity in a free group, we can phrase an intrinsic version of Question \ref{q:extrinsic}.
\begin{ques}[The Intrinsic (Restricted) Primitive Torsion Problems]
\label{q:intrinsic}
Fix positive integers $r$ and $k$. Let $\Gamma$ be a (residually finite) group of rank $r$ such that every primitive element has order dividing $k$. Which questions from Question \ref{q:extrinsic} have affirmative answers?
\end{ques}

The Primitive Torsion Problems are natural variants of the original  \emph{Bounded Burnside Problem}. There has been great progress in understanding the quotients arising from these problems, see for instance \cite{CG17}.
Moreover, studying laws other than the power law in restricted Burnside Problems is a very active area, see \cite{BT2017} and \cite{MR3451381} for the state of the art.

\begin{ques}[The Bounded Burnside Problem]
Fix $r,k \in \Z$.
Let $G$ be a group generated by $r$ elements.
Let $B_k$ be the group in $G$ generated by elements of the form $g^k$ where $g \in G$.
Is $G/B_k$ necessarily finite?
\end{ques}

We note that when $G/P_k$ is virtually solvable, the resulting group $G/B_k$ is necessarily finite. Thus, our work recovers the well-known result that $F_2/B_4$ is finite.
If our conjecture that $F_2/P_5$ is virtually solvable is correct, then it follows that $F_2/B_5$ is finite, which is unknown.

\subsection*{Outline of article}

In \S \ref{sect:generators}, we describe normal generators for $P_k$. We produce finite lists
of normal generators in case $k \in \{2,3,4,5\}$. The generators in these cases correspond to the vertices of the triangular dihedron, the tetrahedron, the octahedron, and the icosahedron.
We use our list of generators to show that the quotients $F_2/P_2$ and $F_2/P_3$ are as listed
in the introduction. Running out of platonic solids with triangular faces, our techniques would give a infinite collection of normal generators for $P_k$ for $k \geq 6$, and so we conjecture that $F_2/P_k$ is not finitely presented for $k \geq 6$. 

In \S \ref{sect:representations}, we produce highly symmetric representations of $F_2/P_k$ into $\GL(n,\C)$ with infinite image when $k \geq 4$. Our technique involves deforming a representation into $\GL(n,\C)$ inside a bigger group, namely $\GL(N,\C)$ for $N>n$. 
We take highly symmetric representations of $F_2/P_k \to \GL(n,\C)$ which factor through a finite group and then deform them in such a way that the representations develop an infinite image in $\GL(N,\C)$ while remaining highly symmetric. This allows us to prove that $F_2/P_k$ is infinite for $k \geq 4$. Also, the process leads to new highly-symmetric representations of $F_2/P_k$. In the case of $k=4$, we repeat this process twice (with a tensor product in the middle) to produce the representation $F_2/P_4 \to \GL(9,\C)$
which was mentioned in the introduction.

In \S \ref{sect:P4}, we prove that our representation $F_2/P_4 \to \GL(9,\C)$ is faithful and proves $F_2/P_4$ has the form mentioned in the introduction.

Appendix \ref{appendix} discusses the relationship between this work and the geometry of square-tiled surfaces.

\input{gen}

\input{representations}

\section{Characterizing $F_2/P_4$}
\label{sect:P4}

A \emph{polycyclic group} is a group that admits a subnormal series with cyclic factors. Any group that is virtually nilpotent is polycyclic.
The \emph{Hirsch length} of a polycyclic group is the number of infinite factors in any subnormal series with cyclic factors.
For any polycyclic group $G$, we will refer to the Hirsch length as the \emph{dimension} of the group, and denote it by $\dim(G)$.
A fact that we will use repeatedly in this section is that for any normal subgroup $N \subset G$, we have (see, for instance, \cite[Theorem 4.7]{MR3729243} for a proof)
$$
\dim(G) - \dim(N) = \dim(G/N).
$$
In particular, if $G$ is torsion-free and $N$ is non-trivial, then it follows that $\dim(G) > \dim(G/N)$.

For this section let $G=F_2/P_4$. 
The following proposition tells us that $G$ is virtually torsion-free nilpotent of dimension equal to 5. 
We will use this result to prove that the representation $\tilde{\tilde \rho}_4$ is faithful. For this, we will need to record the generators of the torsion free subgroup we find.

\begin{prop}
\label{prop:N_p}
Let $N$ be the subgroup of $G$ generated by
$$
a_1=E^{-2}, \quad a_2=A^4  D^2, \quad a_3=A E^{-2} A^{-4} B A^{-1} B^{-1}, \quad \text{and} \quad 
a_4=A^9 C A^{-1} C^{-1} B A B^{-1} A^{-1},
$$
where
$$
A = ba^{-1} b^{-1} a, \quad B = b^{-1} a b a^{-1}, \quad C = b^{-1} a^{-1} b a,$$
$$D = a^2 b (a^{-1} b^{-1})^2 a^{-1} b a, \quad \text{and} \quad
 E = b^{-1} (ab)^2 a^{-3} b^{-1} a.$$
Then $N$ is a 5-dimensional torsion-free nilpotent subgroup of index $2^{12}$ in $F_2/P_4$ that is isomorphic to $H(\Z) \times H(\Z)$ with one non-trivial added relator.
\end{prop}

\begin{proof}
Let $G_2$ be the second term of the derived series for $G$, where $G$ is described in terms of the relations provided by the table on page \pageref{table:P4}. 
Using \cite{GAP4}, we can confirm that $G_2$ is a subgroup of finite index in $G$. Moreover, \cite{GAP4} gives us the following presentation for $G_2$ (the $F_i$ notation follows GAP's output):

\begin{eqnarray*}
\langle F_1, F_2, F_3, F_4, F_5 ~|~  \\
F_3^{-1} F_1^{-1} F_3 F_1 &=&
F_2^{-1} F_3^{-1} F_2 F_3= 
F_2^{-1}F_4 F_2 F_4^{-1}= 1, \\
F_1^{-1} F_2 F_1 F_2^{-1}&=&
F_4 F_5 F_4^{-1} F_5^{-1}= 
F_5^{-1} F_2 F_5 F_2^{-1}= 1, \\
F_4 F_1^{-1} F_5 F_4^{-1} F_1 F_5^{-1}&=& 
F_4^{-1} F_5^{-1} F_3 F_5 F_4 F_3^{-1} =
F_5^{-1} F_3 F_1 F_5 F_3 F_1^{-1} = 1, \\
F_4 F_1 F_3 F_4^{-1} F_1^{-1} F_3^{-1}&=&
F_3^{-1} F_2 F_1^{-1} F_4 F_3 F_2^{-1} F_1 F_4^{-1}=1\\
F_5^{-1} F_3 F_5 F_3^{-1} F_5 F_3 F_5^{-1} F_3^{-1}&=& 1, \\
F_1^{-1} F_2^{-1} F_5 F_3^{-1} F_5^{-1} F_1 F_2^{-2} F_3 F_2^{-1} &=& 
F_2 F_4 F_1 F_4^{-1} F_2^3 F_1^{-1}=1  \rangle
\end{eqnarray*}

From this presentation and computations in \cite{GAP4}, we see that $G_2$ satisfies the following:
\begin{enumerate}
\item first homology of $G_2$ is $\Z/4 \times \Z^4$.
\item $G_2$ has index 1024.
\end{enumerate}

Let $N$ be the group generated by $F_1, F_3, F_4, F_5$.
Using \cite{GAP4}, we can check that $N$ has index $2^{12}$ and has the desired generators.
Moreover, \cite{GAP4} gives that $N$ has a presentation of the form:
$$
N = \left< a_1, a_2, a_3, a_4 ~|~ R \right>,
$$
where $$R = \{ [a_1, a_2], [a_3, a_4], [a_1, a_3], [a_4^2 a_3, a_1^{-1} a_2], [a_2, a_4], [a_4 a_3, a_1^{-2} a_2], [a_4 a_3, (a_1^{-1} a_2)^{-1} a_4 (a_1^{-1} a_2) \}.$$
This is a quotient of the right-angled Artin group $F_2 \times F_2$ with the three added relators 
$$
[a_4^2 a_3, a_1^{-1} a_2], [a_4 a_3, a_1^{-2} a_2], [a_4 a_3, (a_1^{-1} a_2)^{-1} a_4 (a_1^{-1} a_2)].
$$
Viewing the group as $F_2 \times F_2 = \left< a_1, a_4 \right> \times \left< a_2, a_3 \right>$, we can simplify the relations to:
$$
([a_4^2, a_1^{-1}], [a_3, a_2]),
([a_4, a_1^{-2}], [a_3, a_2]),
\gamma := ([a_4, a_1 a_4 a_1^{-1}],1).
$$
Using suitable conjugations, we further simplify the relations to:
$$
([a_1,a_4^2], [a_3, a_2]),
([a_1^2, a_4], [a_3, a_2]),
\gamma := ([a_4, a_1 a_4 a_1^{-1}],1).
$$

Then $N$ is the group $(F_2 \times F_2)/ K$ where $K$ is the normal subgroup generated by the elements above. 

The last relator gives $[a_1, a_4]$ and $a_4$ commute.
By the two other relators, we have
$([a_1, a_4^2],1) = ([a_1^2, a_4],1)$. This equality is equivalent to
$[a_1, a_4] [a_1, a_4]^{a_4} = [a_1, a_4]^{a_1} [a_1, a_4]$.
Hence, $([a_1, a_4], 1)$ is central in $N$.
Moreover, since $([a_1,a_4^2], [a_3, a_2])$ is a relator, $(1, [a_3, a_2])$ is also central in $N$.

Let $H_1$ be the image of $F_2 \times 1$ in $N$ and $H_2$ the image of $1 \times F_2$.
Since $([a_1, a_4],1)$ and $(1, [a_3, a_2])$ are central, the groups $H_1$ and $H_2$ are both quotients of $H(\Z)$ (in fact, they are both isomorphic to $H(\Z)$).
It follows that $N$ must be $H(\Z) \times H(\Z)$ with a relation identifying the square of a central generator of $H(\Z) \times 1$ with one of $1 \times H(\Z)$.
It is now clear that $N$ has infinite center.
Thus, $N$ has Hirsch length equal to 5 and is torsion-free, as desired.
\end{proof}

Recall the definition of $\tilde{\tilde \rho}_4:F_2 \to \GL(9,\C)$ described by \eqref{eq:ttrho1} and \eqref{eq:ttrho2}.
From Proposition \ref{prop:ttrho}, $P_4 \subset \ker \tilde{\tilde \rho}_4$ thus we can consider
$\tilde{\tilde \rho}_4$ to be a homomorphism from $G$ to $\GL(9,\C)$.

\begin{thm}
\label{thm:faithful}
The representation $\tilde{\tilde \rho}_4:F_2/P_4 \to \GL(9,\C)$ is faithful. We have $d=1$ in Proposition \ref{prop:ttrho}.
\end{thm}
\begin{proof}
First we will show that $d=1$ using a dimension argument. For this proof, consider $\rho_4$, $\tilde \rho_4$ and $\tilde{\tilde \rho}_4$ to homomorphisms from $F_2/P_4$ and their kernels to be subgroups of $F_2/P_4$.
We can compute that the generators of $N$ lie in $\ker \rho_4$, and we conclude $N \subset \ker \rho_4$. Propositions \ref{prop:trho4}
and \ref{prop:ttrho} tell us that $\tilde \rho_4(\ker \rho_4)$ is isomorphic to $\Z^4$ and 
$\tilde {\tilde \rho}_4(\ker \rho_4)$ is a further $\Z^d$-extension for $d \geq 1$. It follows that $\tilde {\tilde \rho}_4(\ker \rho_4)$ is polycyclic. Moreover, $\dim \tilde {\tilde \rho}_4(\ker \rho_4) = 4+d \geq 5$. Since dimension non-strictly drops under surjective homomorphisms, we have
$\dim N \geq \dim \tilde {\tilde \rho}_4(N)$, and since $N$ is finite index inside of $F_2/P_4$, we have
$\dim \tilde {\tilde \rho}_4(N)=\dim \tilde {\tilde \rho}_4(F_2/P_4)$. Putting this all together we have
$$5 = \dim N \geq \dim \tilde {\tilde \rho}_4(N) = \dim \tilde {\tilde \rho}_4(F_2/P_4) = 4+d \geq 5.$$
We conclude that all expressions in the above line are $5$, and therefore $d=1$. Since non-trivial quotients of $N$ have strictly smaller dimension, we also get that the restriction of $\tilde {\tilde \rho}_4$ to $N$ is injective. Thus the faithfulness claimed in the theorem will follow if we can prove that subgroup indices satisfy
$$[\tilde {\tilde \rho}_4(F_2/P_4):\tilde {\tilde \rho}_4(N)] = [F_2/P_4:N].$$
We already know that $[F_2/P_4:N]=2^{12}$. It suffices to prove that $[\tilde{\tilde \rho}_4(G):\tilde{\tilde \rho}_4(N)] \geq 2^{12}$ since index can not grow under group homomorphisms.

First observe that $[\rho_4(G):\rho_4(N)]=2^3$ since $N \subset \ker \rho_4$ and $\rho_4(G)$ is isomorphic to the quaternion group. 

Now consider the index $[\tilde \rho_4(G):\tilde \rho_4(N)]$. Let $a_1, a_2, a_3, a_4$ denote the generators for $N$ listed in Proposition \ref{prop:N_p}.
Define $\gamma:\Z[i]^2 \to \GL(4,\C)$ as in \eqref{eq:gamma}.
By Proposition \ref{prop:trho4}, $\tilde \rho_4(\ker \rho_4)=\gamma(\Lambda)$ where $\Lambda \subset \Z[i]^2$ is a subgroup of index two. We compute
\begin{equation}
    \label{eq:tilde rho g}
\begin{array}{lcl}
\tilde \rho_4(a_1)=\gamma(-2i-2,-2i+2), & & \tilde \rho_4(a_2)=\gamma(2i-2,2i+2) \\
\tilde \rho_4(a_3)=\gamma(4,0), & \text{and} & \tilde \rho_4(a_4)=\gamma(0,4i). \\
\end{array}
\end{equation}
Thus $\tilde \rho_4(N)=\gamma(\Lambda')$ where
$$\Lambda'=\langle 
(-2-2i,-2i+2),(2i-2,2i+2),(4,0),(0,4i)\rangle).$$
Based on this, we observe $\Lambda' \subset \Lambda$ and we can compute that $\big[\Z[i]^2:\Lambda'\big]=2^7$ and thus $[\Lambda:\Lambda']=2^6$. It follows that 
$$[\tilde \rho_4(\ker \rho_4):\tilde \rho_4(N)]=2^{6} \quad \text{and} \quad
[\tilde \rho_4(F_2/P_4):\tilde \rho_4(N)]=2^{6+3}.$$

Finally, we consider the index $[\tilde{\tilde \rho}_4(G):\tilde{\tilde \rho}_4(N)]$. From the above, we know that $F_2/\ker \tilde{\tilde \rho}_4$ is a $\Z$-extension of $F_2/\ker \tilde \rho_4$. We have $[a,b]^2 \in \ker \tilde \rho_4$ but $\tilde{\tilde \rho}_4([a,b]^2) \neq I$ (see \eqref{eq:image of commutator squared}).
Since the images under $\tilde \rho_4$ of the four generators $a_i$ freely generate the image $\tilde \rho_4(N)$ which is isomorphic to $\Z^4 \cong N / [N, N]$, 
it follows that $N \cap \ker \tilde \rho_4 = [N, N]$. Since $N$ is two-step nilpotent, this commutator subgroup is generated by commutators of the generators of $N$. We compute
$$\tilde{\tilde \rho}_4([a_1,a_2])=\tilde{\tilde \rho}_4([a_3,a_4])=I.$$
For other pairs of generators of $N$ we have:
$$
\tilde{\tilde \rho}_4([a_3,a_1])=
\tilde{\tilde \rho}_4([a_4,a_1])=
\tilde{\tilde \rho}_4([a_2,a_3])=
\tilde{\tilde \rho}_4([a_4,a_2])=
\tilde{\tilde \rho}_4([a,b]^2)^8.$$
Thus the central copy of $\Z$ in $\tilde{\tilde \rho}_4(F_2/P_4)$ contains $\tilde{\tilde \rho}_4(N \cap \ker \tilde \rho_4)$ with index at least $2^3$. Consequently, $[\tilde{\tilde \rho}_4(F_2/P_4):\tilde {\tilde \rho}_4(N)] \geq 2^{3+6+3}$ as desired.
\end{proof}

% We remark that this proof also gives that $[a,b]^2$ generates the center of $F_2/P_4$.

%
%
% \subsection{Describing a quotient of $P_4$ through explicit covers}
%
% Two covers of the figure eight are indicated in Figure \ref{fig:covers}. One explicitly shows that the they form an automorphism class in the set of subgroups in $F_2$.
% Hence, the intersection of the two subgroups, call it $K$, corresponding to these two covers is characteristic.
% Moreover, both subgroups contain all the relations for $P_4$.
% It follows that $K \geq P_4$.5
%
% \begin{figure}
% \includegraphics[width=\textwidth]{fourthcover}
% \caption{Two infinite covers of the figure eight.}
%  \compat{The path associated to $(ab)^4$ does not seem to close up in the diagram at right.}
%  \comkha{I'm sorry, I copied it from my notes incorrectly. It's fixed now.}
% \label{fig:covers}
% \end{figure}
%

%\subsection{Explicit representations of $F_2/P_4$}

\appendix 

\input{square-tiled}

\subsection*{Acknowledgements}

KB was partially supported by the National Science Foundation under Grant Number DMS-1405609 as well as by a PSC-CUNY Award (funded by The Professional Staff Congress and The City University of New
York). PH was partially supported by the National Science Foundation under Grant Number DMS-1500965,
by the Simons Foundation under Grant \#632227,
and by a PSC-CUNY Award (funded
by The Professional Staff Congress and The City University of New
York). The authors are grateful to Benson Farb, Ilya Kapovich, Thomas Koberda, Justin Malestein, Andrew Putman, R\'emi Coulon, Dominik Gruber, and Jack Oliver Button for useful conversations, and grateful to an anonymous referee for numerous helpful comments.

\bibliography{refs}
\bibliographystyle{alpha}

\end{document}

%% file: gen.tex
\section{Normal generators for \texorpdfstring{$P_k$}{P\_k}}
\label{sect:generators}

\subsection{Primitive elements of \texorpdfstring{$F_2$}{the free group of rank two}}
\label{sect:primitive}

Let $F_2$ denote the free group $\langle a, b\rangle$. The reader will recall or quickly observe the following facts about primitive elements of $F_2$:
\begin{enumerate}
\item If $c \in F_2$ is primitive then there is a $\phi \in \Aut(F_2)$ such that $\phi(a)=c$. 
\item If $c \in F_2$ is primitive then so every element of its conjugacy class $[c]=\{gcg^{-1}:~g \in F_2\}$. 
\end{enumerate}
In particular, we will say a conjugacy class is {\em primitive} if it consists of primitive elements of $F_2$.

The observation that there is a short exact sequence
$$1 \to F_2 \to \Aut(F_2) \xrightarrow{D} \GL(2,\Z) \to 1$$
dates back to Jakob Nielsen's 1913 Thesis. Here the map $F_2 \to \Aut(F_2)$ sends an element of $F_2$ to its corresponding inner automorphism and thus $\GL(2,\Z)$ is isomorphic to the outer automorphism group $\Out(F_2)=\Aut(F_2)/\Inn(F_2)$.
The map $D:\Aut(F_2) \to \GL(2,\Z)$ may be defined by using the abelianization homomorphism 
$\mathrm{ab}:F_2 \to \Z^2$, which we choose to satisfy $a \mapsto (1,0)$ and $b \mapsto (0,1)$. Then $D(\phi) \in \GL(2,\Z)$ is determined by the condition that $D(\phi) \circ \mathrm{ab}(g)=\mathrm{ab} \circ \phi(g)$ for all $g \in F_2$. 

An automorphism of $F_2$ either preserves the conjugacy class of the commutator $[a,b]$ or sends it to the conjugacy class of $[b,a]$. Thus there is a natural homomorphism $\Aut(F_2) \to C_2$ where we identify $C_2$ with the permutation group of these conjugacy classes. We set $\Aut_+(F_2)$ to be the kernel which consists of automorphisms preserving the conjugacy class of the commutator $[a,b]$. We use $\Aut_-(F_2)$ to denote $\Aut(F_2) \smallsetminus \Aut_+(F_2)$.

The group $\Out_+(F_2)=\Aut_+(F_2)/\Inn(F_2)$ is isomorphic to $\SL(2,\Z)$ via the map $D$ above. The following elements
of $\Aut_+(F_2)$ have images in $\Out_+(F_2)$ which generate:
\begin{equation}
\label{eq:SL2Z generators}
\psi_0(a)=b,~\psi_0(b)=b^{-1}a^{-1}; \quad 
\psi_1(a)=b,~\psi_1(b)=a^{-1}; \quad 
\psi_2(a)=a,~\psi_2(b)=ab.
\end{equation}
We will use $\bar \psi_0$, $\bar \psi_1$ and $\bar \psi_2$ to denote the outer automorphism classes of these elements.
It may be observed that the following identities are satisfied:
\begin{equation}
\label{eq:psi relations}
\psi_0 \circ \psi_2=\psi_1, \quad \psi_0^3=\psi_1^4=1, \quad [\bar \psi_1^2,\bar \psi_0]=[\bar \psi_1^2,\bar \psi_2]=1.
\end{equation}

Recall that outer automorphisms act on conjugacy classes. We will use $[g]$ to denote the conjugacy class of $g \in F_2$. We have the following:
\begin{lemma}[Primitive conjugacy classes]
\label{lem:primitive conjugacy classes}
An element $g \in F_2$ is primitive if and only if it lies in the conjugacy class $\bar \psi([a])$ for some $\bar \psi \in \Out_+(F_2)$. 
\end{lemma}
\begin{proof}
If $g \in F_2$ is primitive then by (1) above there is a $\psi \in \Aut(F_2)$ such that $\psi(a)=g$. Then
by possibly precomposing with the automorphism $\psi_- \in \Aut_-(F_2)$ determined by $\psi_-(a)=a$ and $\psi_-(b)=b^{-1}$ we can assume that $\psi \in \Aut_+(F_2)$. Let $\bar \psi \in \Out_+(F_2)$ be the class containing $\psi$. Then $[g]=\bar \psi([a])$. 
The converse is clear since primitivity is a conjugacy invariant and is invariant under automorphisms.
\end{proof}

It follows that the conjugacy classes of primitive elements are naturally identified with $\Out_+(F_2)$ modulo the stabilizer
of the conjugacy class $[a]$. This stabilizer is $\langle \bar \psi_2 \rangle$.

The primitive conjugacy classes come naturally in pairs: if $g \in F_2$ is primitive, then we call the conjugacy classes $[g]$ and $[g^{-1}]$ {\em opposites}. We will denote the collection of unions of opposite pairs
of conjugacy classes by ${\mathcal P}$. Opposites are related by the action of the central involution $\bar \psi_1^2$ of $\Out_+(F_2)$:

\begin{prop}
If $[g]$ is a primitive conjugacy class then its opposite $[g^{-1}]$ is $\bar \psi_1^2([g])$. 
\end{prop}
\begin{proof}
From the lemma above we have $[g]=\bar \psi([a])$ for some $\bar \psi \in \Out_+(F_2)$. 
Since $\psi_1^2(a)=a^{-1}$ we have $\bar \psi_1^2([a])=[a^{-1}]$ and  $[g^{-1}]=\bar \psi \circ \bar \psi_1^2([a])$. Since $\bar \psi_1^2$ is central in $\Out_+(F_2)$ we have $[g^{-1}]=\bar \psi_1^2 \circ \bar \psi([a])=\bar \psi_1^2([g])$. 
\end{proof}

Since $\langle \bar \psi_2 \rangle$ is the stabilizer of $[a]$ and $\bar \psi_1^2$ acts as above,
there is a bijective correspondence from the coset space
\begin{equation}
\label{eq:C to P}
\mathcal{C}= \Out_+(F_2) / \langle \bar \psi_2, \bar \psi_1^2 \rangle
\quad \text{to ${\mathcal P}$ given by} \quad
\bar \psi \langle \bar \psi_2, \bar \psi_1^2 \rangle \mapsto \bar \psi([a]) \cup \bar \psi([a^{-1}]).
\end{equation}

The group $\SL(2,\Z)/\pm I$ has a well known action on the upper half plane by M\"obius transformations with $-I$ acting trivially. Here the matrix
$$\left(\begin{array}{rr}
m_{11} & m_{12} \\
m_{21} & m_{22}
\end{array}\right) \quad \text{acts by} \quad z \mapsto \frac{m_{11}z+m_{12}}{m_{21}z+m_{22}}.$$
This is useful for organizing the pairs of primitive conjugacy classes. Observe that $\langle D(\bar \psi_2), D(\bar \psi_1)\rangle$ is the stabilizer in $\SL(2,\Z)$ of the point $\frac{1}{0}$. The $\SL(2,\Z)$ orbit of $\frac{1}{0}$ is $\hat \Q=\Q\cup \{\frac{1}{0}\}$. Thus, we have:

\begin{lemma}
\label{lem:maps to Q}
There are bijections ${\mathcal C}:\hat \Q \to {\mathcal C}$ and ${\mathcal P}:\hat \Q \to {\mathcal P}$ compatible with \eqref{eq:C to P} such that for any $\frac{p}{q} \in \hat Q$ we have:
\begin{itemize}
\item The class ${\mathcal C}(\frac{p}{q})$ is the collection of $\bar \psi \in \Out_+(F_2)$ such that $D(\bar \psi) (\frac{1}{0})=\frac{p}{q}$. 
\item The union of the pair of conjugacy classes ${\mathcal P}(\frac{p}{q})$ consists of all primitive elements $g \in F_2$ such that $\mathrm{ab}(g)=\pm(p,q)$ (where $p,q\in \Z$ are taken to be relatively prime).
\end{itemize}
\end{lemma}

The {\em Farey triangulation} ${\mathcal F}$ is an $\SL(2,\Z)$ invariant triangulation of the upper half plane with vertices in $\hat \Q$. We depict ${\mathcal F}$ in Figure \ref{fig:farey}. The group $\PSL(2,\Z)$ is the orientation preserving symmetry group of ${\mathcal F}$. It is useful to think of the three spaces $\hat \Q$, ${\mathcal C}$ and ${\mathcal P}$ as in bijective correspondence to the vertices in this triangulation.

\begin{figure}
\includegraphics[width=\textwidth]{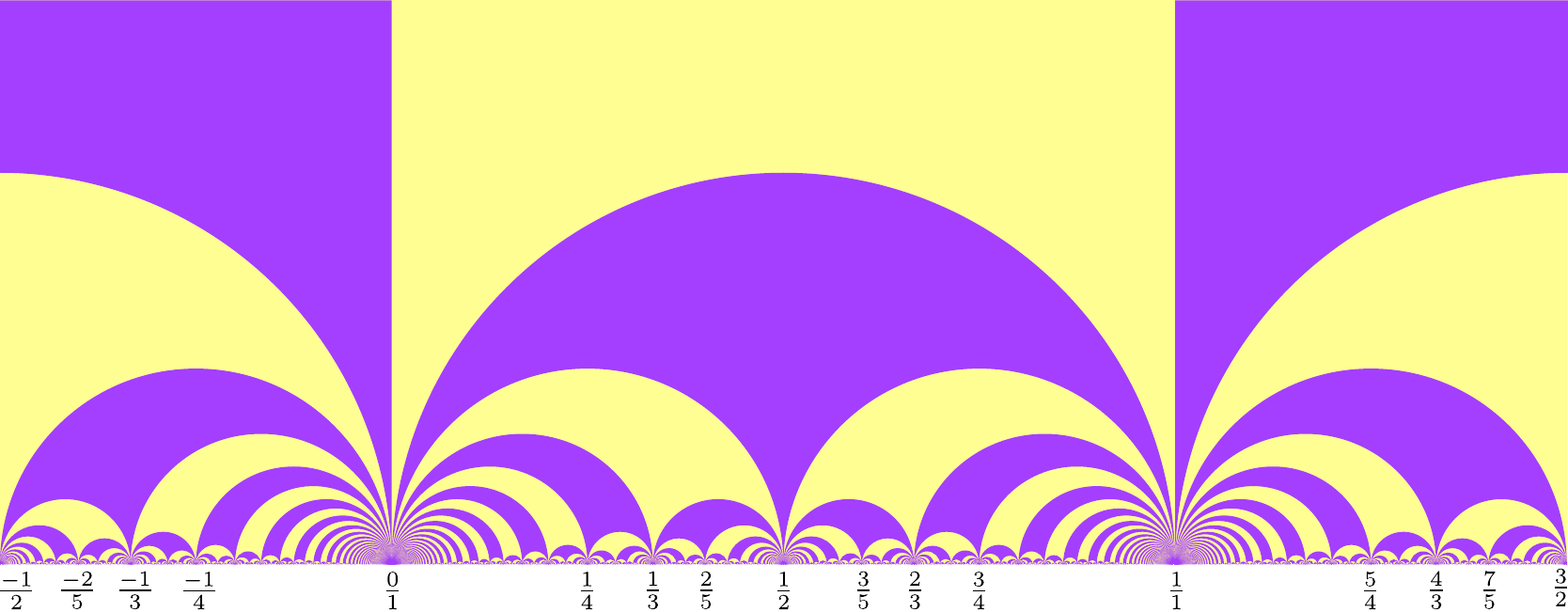}
\caption{A portion of the Farey triangulation ${\mathcal F}$ of the hyperbolic plane with some rational points at infinity marked. The top endpoint of vertical edges is $\frac{1}{0}$.}
\label{fig:farey}
\end{figure}

\subsection{Symmetries and images of primitive elements}

Having a power of a primitive element in a normal subgroup $N$ guarantees that some corresponding elements of $\Out(F_2)$ stabilize $N$. It suffices to consider the case when a power of $a$ lies in $N$.

\begin{lemma}
\label{lem:multitwist}
Suppose $N \subset F_2$ is a normal subgroup containing $a^k$ for some $k \geq 1$. Then $\psi_2^k(N)=N$. Furthermore, the induced action of $\psi_2^k$ on $F_2/N$ given by 
$$hN \mapsto \psi_2^k(h)N$$
is trivial.
\end{lemma}

\begin{proof}
Assume $a^k \in N$ and $N \subset F_2$ is normal. Observe that the action of $\psi_2^k$ satisfies
$$\psi_2^k(a)=a, \quad \psi_2^k(a^{-1})=a^{-1}, \quad \psi_2^k(b)=a^k b \quad \text{and} \quad
\psi_2^k(b^{-1})=b^{-1} a^{-k}.$$
Let $h \in F_2$ and consider $h$ as a word in $\{a,a^{-1},b,b^{-1}\}$. From the above  description of $\psi_2^k$ we see that $\psi_2^k(h)$ is formed from $h$ by inserting copies of $a^k$ and $a^{-k}$ into the word representing $h$. Let $n$ be the number of such insertions. Then we can write 
$$h=\psi_2^k(h) g_1 g_2 \ldots g_n$$
where each $g_i$ is a conjugate of either $a^{-k}$ or $a^{k}$ selected to remove an inserted copy of $a^k$ or $a^{-k}$.
Since $a^k \in N$ and $N$ is normal, each $g_i \in N$.
It follows that $h \in N$ if and only if $\psi_2^k(h) \in N$. Thus $\psi_2^k(N)=N$. Finally we see that for any $hN \in F_2/N$,
$$\psi_2^k(hN) = \psi_2^k(h)N = h g_n^{-1} \ldots g_2^{-1} g_1^{-1} N=hN.$$
\end{proof}

We get the following if a power of a primitive element lies in a normal subgroup of $F_2$.

\begin{cor}
\label{cor:arbitrary primitive power}
Let $\frac{p}{q} \in \hat Q$, let $g \in {\mathcal P}(\frac{p}{q})$, and let $\psi:F_2 \to F_2$ be an automorphism such that the associated outer automorphism
$\bar \psi$ lies in ${\mathcal C}(\frac{p}{q})$.
The for any $k \geq 2$ and for any normal subgroup $N' \subset F_2$ containing $g^k$, 
$\psi \circ \psi_2^k \circ \psi^{-1}(N')=N'$ and the induced action of $\psi \circ \psi_2^k \circ \psi^{-1}$ on $F_2/N'$ is trivial.
\end{cor}
\begin{proof}
From Lemma \ref{lem:maps to Q}, we note that $\bar \psi([a] \cup [a^{-1}])=[g] \cup [g^{-1}]$. Set $N=\bar \psi^{-1}(N')$. By normality, we see that $a^k \in N$. 
Thus Lemma \ref{lem:multitwist} tells us that $\psi_2^k(N)=N$ and $\psi_2^k$ acts trivially on $F_2/N$. It follows that $\psi \circ \psi_2^k \circ \psi^{-1}$ stabilizes $N$ and acts trivially on $F_2/N'$.
\end{proof}

Recall that $P_k \subset F_2$ is the subgroup generated by $k$-th powers of primitive elements of $F_2$. This subgroup is clearly characteristic, thus there is a well defined homomorphism
$$\epsilon: \Aut(F_2) \to \Aut(F_2/P_k); \quad \epsilon(\phi)(g P_k)=\phi(g) P_k.$$
Inner automorphism of $F_2$ are sent by $\epsilon$ to  inner automorphisms of $F_2/P_k$, thus $\epsilon$ induces a well defined map between outer automorphism groups
$$\bar \epsilon: \Out(F_2) \to \Out(F_2/P_k).$$
Let ${\mathcal O}_k \subset \Out_+(F_2)$ denote the subgroup normally generated by $\bar \psi_2^k$. The lemma above guarantees:

\begin{cor}
The subgroup ${\mathcal O}_k$ is contained in $\ker \bar \epsilon$.
\end{cor}
\begin{proof}
We must show that for each $\bar \psi \in \Out_+(F_2)$ we have $\bar \psi \circ \bar \psi_2^k \circ \bar \psi^{-1} \in \ker \bar \epsilon$. 
Since $\ker \bar \epsilon$ is a normal subgroup, we may take $\bar \psi = 1$, and that $\bar \psi_2^k \in \ker \bar \epsilon$ follows from Lemma \ref{lem:multitwist}.
% old:
%Fixing $\bar \psi$, we may find a $\frac{p}{q} \in \hat \Q$ so that $\bar \psi \in {\mathcal C}(\frac{p}{q})$. Choose a $g \in {\mathcal P}(\frac{p}{q})$. Then $g$ is primitive so that $g^k \in P_k$. Choose a representative $\psi \in \Aut_+(F_2)$ of $\bar \psi$. Lemma \ref{lem:multitwist} guarantees that $\epsilon(\psi \circ \psi_2^k \circ \psi^{-1})$ is the trivial automorphism of $F_2/P_k$ so that $\bar \psi \circ \bar \psi_2^k \circ \bar \psi^{-1} \in \ker \bar \epsilon$ as desired.
\end{proof}

% Old Version (referee objected...)
%We remark that the image under $D$ of $\bar \psi_2^k$ has the effect of ``rotating'' the Farey triangulation ${\mathcal F}$ by $k$ triangles about the vertex $\frac{1}{0}$. It follows that the triangulated space
%\begin{equation}
%\label{eq:F_k}
%{\mathcal F}_k={\mathcal F} / D {\mathcal O}_k
%\end{equation}
%is combinatorially the triangulation of a simply connected surface with $k$ triangles meeting at every vertex. So, ${\mathcal F}_k$ a triangulation of the sphere when $k\leq 5$, combinatorially equivalent to the tiling of the plane by Euclidean triangles  when $k=6$, and combinatorially equivalent to a tiling of the hyperbolic plane by regular triangles when $k \geq 7$. 

In order to better understand ${\mathcal O}_k$ we make use of $D:\Out_+(F_2) \to \SL(2,\Z)$ and the M\"obius action on the  Farey triangulation ${\mathcal F}$. Note that $\SL(2,\Z)$ is the group of orientation preserving symmetries of ${\mathcal F}$ which permute the triangles. Thus covering space theory identifies each subgroup $\Gamma \subset \SL(2,\Z)$ bijectively with the (possibly orbifold) quotient ${\mathcal F}/\Gamma$ which is tiled by triangles (possibly including some quotients of triangles by their order $3$ rotation groups). These quotients are intermediate between ${\mathcal F}$ and the modular surface ${\mathcal F}/\SL(2,\Z)$ (which has a vertex added at the cusp since ${\mathcal F}$ includes vertices). The {\em valence} of a vertex in a triangulation is the number of vertices of triangles that are identified to make that point. The valence of a vertex may be a positive integer or infinity.

The following gives a concrete understanding of the quotient ${\mathcal F}_k={\mathcal F}/D{\mathcal O}_k$:

\begin{prop}
\label{prop:F_k}
The orbifold ${\mathcal F}_k$ is the unique simply connected triangulated surface such that all vertices have valence $k$. In particular, the combinatorial type of the triangulated surface ${\mathcal F}_k$ can be described as follows:
\begin{itemize}
    \item If $k \in \{2,3,4,5\}$, then ${\mathcal F}_k$ is a sphere. Specifically ${\mathcal F}_2$ is a triangle doubled across its boundary, ${\mathcal F}_3$ is a tetrahedron, ${\mathcal F}_4$ is an octahedron, and ${\mathcal F}_5$ is an icosahedron.
    \item The quotient ${\mathcal F}_6$ is the plane tiled by equilateral triangles.
    \item For $k \geq 7$, the quotient ${\mathcal F}_k$ is the hyperbolic plane tiled by equilateral triangles each of whose angles measures $\frac{2\pi}{k}$.
\end{itemize}
\end{prop}
\begin{proof}
First observe that $D\psi_2$ acts as the M\"obius transformation $z \mapsto z+1$, and thus sends each triangle of ${\mathcal F}$ incident to $\infty$ to the adjacent triangle in the counterclockwise direction about $\infty$. Thus if $\Gamma \subset \SL(2,\Z)$ contains $D \psi_2^k$, the corresponding quotient ${\mathcal F}/\Gamma$ has valence dividing $k$ at the vertex in the image of $\infty$ under the covering ${\mathcal F} \to {\mathcal F}/\Gamma$.

Now suppose $\Gamma$ contains all of $D{\mathcal O_k}$. Since $\SL(2,\Z)$ acts transitively on $\hat \Q$ and $D{\mathcal O_k}$ is normal in $\SL(2,\Z)$, it follows that each vertex of ${\mathcal F}/\Gamma$
has valence dividing $k$.

Now consider moving from orbifolds to groups. Let $S$ be a connected combinatorial orbifold built by identifying in pairs the edges of some collection of triangles and quotients of a triangle modulo the order three rotation. Such an $S$ is covered by the Farey triangulation, and fixing such a covering map $\pi:{\mathcal F} \to S$, covering space theory associates the deck group
$$\Gamma=\{M \in \SL(2,\Z):~\pi \circ M=\pi\}.$$
We observe that if each vertex of $S$ has valence dividing $k$, then $D{\mathcal O}_k \subset \Gamma$.

We conclude from the previous paragraph that the quotients of ${\mathcal F}$ described in the proposition are of the form ${\mathcal F}/\Gamma$ for some
$\Gamma$ containing $D{\mathcal O}_k$. To see $\Gamma=D{\mathcal O}_k$ recall from covering space theory, the surface ${\mathcal F}_k$ (branched) covers any ${\mathcal F}/\Gamma$ with $D{\mathcal O}_k \subset \Gamma$. But, since the surfaces described in the proposition are simply connected and have all vertices of valence precisely $k$, they exhibit no (non-trivial) branched covers such that all vertices of the cover have valence dividing $k$.
\end{proof}

The same mechanism can be used to shorten the list of group elements needed  to normally generate $P_k$. 
\begin{thm}
\label{thm:generators}
Let $k \geq 2$. Let $\{\frac{p_i}{q_i}:~i \in \Lambda\}$ be a subset of $\hat \Q$ containing one representative of each preimage of a vertex of ${\mathcal F}_k$ under the covering map
${\mathcal F} \to {\mathcal F}_k$. For each $i \in \Lambda$ choose a primitive element $g_i \in {\mathcal P}(\frac{p_i}{q_i})$ and an outer automorphism $\bar \psi_i \in {\mathcal C}(\frac{p_i}{q_i})$. If 
$$\{\bar \psi_i \circ \bar \psi_2^k \circ \bar \psi_i^{-1}:~i \in \Lambda\}$$ 
generates ${\mathcal O}_k$ then $P_k$ is normally generated by
$\{ g_i^k:~i \in \Lambda\}$.
\end{thm}
\begin{proof}
Fix the quantities above and assume all hypotheses are satisfied. Let $Q$ be the subgroup of $F_2$ normally generated by $\{ g_i^k:~i \in \Lambda\}$. Clearly $Q \subset P_k$ since each $g_i$ is primitive. We will show $P_k \subset Q$. 

As a consequence of Corollary \ref{cor:arbitrary primitive power} we know that $\bar \psi_i \circ \bar \psi_2^k \circ \bar \psi_i^{-1}$ stabilizes $Q$
for all $i \in \Lambda$. Then from the hypotheses we know each element of ${\mathcal O}_k$ stabilizes $Q$. 

To show $P_k \subset Q$, it suffices to show that if $g \in F_2$ is primitive then $g^k \in Q$. 
Fix $g$. Then there is a $\frac{p}{q} \in \hat \Q$ such that ${\mathcal P}(\frac{p}{q})=[g] \cup [g^{-1}]$. From our hypothesis on $\{ \frac{p_i}{q_i}\}$ we know there is an $i \in \Lambda$ and a $\bar \psi \in {\mathcal O}_k$ such that $D \bar \psi (\frac{p_i}{q_i})=\frac{p}{q}$. Then $\bar \psi([g_i] \cup [g_i^{-1}])=[g] \cup [g^{-1}]$. 
By definition of $Q$ we know that the conjugacy classes $[g_i^k]$ and $[g_i^{-k}]$ are contained in $Q$. Since $Q$ is ${\mathcal O}_k$-invariant and $g^k \in \bar \psi([g_i^k] \cup [g_i^{-k}])$ we have $g^k \in Q$ as desired.
\end{proof}

The following describes a combinatorial way to find the generators:

\begin{cor}
\label{cor:generators}
Fix $k \geq 2$. Let $T \subset {\mathcal F}_k$ be a tree in the $1$-skeleton of ${\mathcal F}_k$ whose vertex set coincides with the collection of all vertices of the triangulation. Let $\tilde T$ be a lift of $T$ to ${\mathcal F}$ and let $\{\frac{p_i}{q_i}:~i \in \Lambda\}$ be the vertices of $\tilde T$. Then $P_k= \langle \langle g_i^k \rangle \rangle_{i \in \Lambda}$ where each $g_i \in {\mathcal P}(\frac{p_i}{q_i})$ is chosen arbitrarily as in Theorem \ref{thm:generators}. 
\end{cor}
\begin{proof}
We must check the hypotheses of Theorem \ref{thm:generators}. 
Define $\{\frac{p_i}{q_i}\}$ and $\{g_i\}$ as in the statement of the corollary and $\{\bar \psi_i\}$ as in Theorem \ref{thm:generators}. 
Since the vertices of $T$ include all vertices of ${\mathcal F}_k$,
we see that $\{\frac{p_i}{q_i}\}$ contains one preimage of each vertex of ${\mathcal F}_k$. 
Let $Q=\langle \bar \psi_i \circ \bar \psi_2^k \circ \bar \psi_i^{-1} \rangle \subset {\mathcal O_k}$. 
We need to show $Q={\mathcal O}_k$. 

Associated to the chain of subgroups $\{1\} \subset  Q \subset {\mathcal O}_k$ is the sequence of spaces related by covering maps branched at the vertices of the triangulations:
$${\mathcal F} \to {\mathcal F}/DQ \xrightarrow{\pi} {\mathcal F}_k.$$
Proving that $Q={\mathcal O}_k$ is equivalent to proving that $\pi$ is the trivial covering. Note that triviality will follow from Proposition \ref{prop:F_k} if all vertices of ${\mathcal F}/DQ$ have valence dividing $k$, so this is what we will prove.

Let $T_Q \subset {\mathcal F}/DQ$ denote the image of $\tilde T$ under the covering map ${\mathcal F} \to {\mathcal F}/DQ$.
Then $T_Q$ is a tree because $\pi(T_Q)=T$. Observe that each vertex of $T_Q$ is incident to $k$ triangles
because such a vertex is the image of some $\frac{p_i}{q_i} \in \tilde T$ and the action of $D(\bar \psi_i \circ \bar \psi_2^k \circ \bar \psi_i^{-1})$ on ${\mathcal F}$ rotates by $k$ triangles about $\frac{p_i}{q_i}$. 
Thus it suffices to prove that every vertex of ${\mathcal F}/DQ$ is a vertex of the tree $T_Q$. 
If this were not the case then there would be an edge of a triangle of ${\mathcal F}/DQ$ with one vertex in $T_Q$ and the other not in $T_Q$. We will show this doesn't happen. 

A key observation is the following.
Say that the {\em link} of a vertex of a triangulated surface is the union of the vertex with the interiors of incident edges and triangles. 
The {\em link lifting observation} is the observation that $\pi$ 
restricted to the link of a vertex $v_Q \in T_Q \subset {\mathcal F}/DQ$ is a bijection to the link of the image vertex $v=\pi(v_Q) \in T \subset {\mathcal F}_k$ since both $v_Q$ and $v$ are incident to $k$ triangles.

Now we return to the proof. Suppose $e_Q=\overrightarrow{v_Q w_Q}$ be an oriented edge of a triangle of ${\mathcal F}/DQ$ initiating at a vertex $v_Q$ of $T_Q$. We will show that the terminating vertex $w_Q$ also is a vertex of $T_Q$. Let $e=\overrightarrow{v w}$ be $\pi(e_Q)$.
We break into two cases. 

First, it could be that $e$ is an edge of $T$. Since $v_Q \in T_Q$ by the link lifting observation we know that $e$ has a unique lift to ${\mathcal F}_Q$ initiating at $v_Q$. Since $T_Q$ is a lift of $T$ and $e$ is an edge of $T$, this means that $e_Q$ must be an edge of $T_Q$. Thus, $w_Q$ is also a vertex of $T_Q$ as desired.

Now suppose that $e$ is not an edge of $T$. 
Since $T$ is a spanning tree, both $v$ and $w$ are vertices of $T$. As $T$ is a tree, there is a unique oriented path $p$ in $T$ joining $v$ to $w$. Let $v=p_0,p_1,\ldots,p_n=w$ be the sequence of vertices passed through by $p$. We will inductively prove $p$ has a unique lift to ${\mathcal F}/DQ$ starting at $v_Q$. This involves checking that for each $j \in \{1,\ldots,n\}$ there is a unique lift of the path $p_0, \ldots, p_j$ denoted $\tilde p_0, \ldots, \tilde p_j$ such that $\tilde p_0=v_Q$ and $\pi(\overrightarrow{\tilde p_i \tilde p_{i+1}})=
\overrightarrow{p_i p_{i+1}}$ for $i \in \{0, \ldots, j-1\}$. This is true for $j=1$ because $v_Q \in T_Q$ using the unique lifting provided by the observation above. Now we will argue the inductive step. Suppose the lift is unique up through index $j<n$. Then since $p$ is a path in $T$ and $\pi(T_Q)=T$, we must have that all vertices of the lift so far lie in $T_Q$. From the link lifting observation we know that there is a unique lift of the next edge $\overrightarrow{p_j p_{j+1}}$ completing the inductive step.

Now observe that since ${\mathcal F}_k$ is a triangulation of a simply connected surface, $p \cup e$ bounds a topological disk $\Delta$.
%Since all vertices of ${\mathcal F}_k$ lie in $T$, none are contained in the interior of $T$. %%% I agree with the referee that this sentence was unnecessary (and I'm not sure what it means...). 
Since the vertex $v_Q \in T_Q$, by the link lifting observation again, there is a unique lift $\tilde \Delta$ of $\Delta$ to ${\mathcal F}/DQ$ such that $v$ lifts to $v_Q$. From the previous paragraph, the path $p$ in the boundary of $\Delta$ lifts to a path $\tilde p$ in the boundary of $\tilde \Delta$ and contained in the tree $T_Q$. Again by the assumption, edge $e$ in the boundary of $\Delta$ lifts to $e_Q$ in the boundary of $\tilde \Delta$. Thus, $e_Q$ joins the initial point $v_Q$ of $\tilde p$ to the terminal point $w_Q$ of $\tilde p$. Since $\tilde p$ is contained in $T_Q$ we see that $w_Q \in T_Q$ as desired.
\end{proof}

\begin{conj}
The normal generators for $P_k$ provided by Corollary \ref{cor:generators} are a minimal set of normal generators. 
In particular, for $k \geq 6$, the group $F_2/P_k$ is not finitely presented.
\end{conj}

\subsection{Normal generators for \texorpdfstring{$P_k$}{P\_k} with \texorpdfstring{$k \leq 5$}{k less than or equal to 5}}
We describe normal generators for $P_k$ when $k\leq 5$ because these are the cases where Corollary \ref{cor:generators} 
yields a finite set of normal generators. These cases are finite because Proposition \ref{prop:F_k} tells us that ${\mathcal F}_k$ is a triangulated sphere.

\subsection*{The case \texorpdfstring{$k=2$}{k=2}}
The triangulated sphere ${\mathcal F}_2$ is the double of a triangle across its boundary. Below we depict a tree $T$ in an unfolding of ${\mathcal F}_2$.
We have lifted $T$ to a tree $\tilde T$ in the Farey triangulation and labeled the vertices of $T$ by their lifts as elements of $\hat \Q$.
Following Theorem \ref{thm:generators} and Corollary \ref{cor:generators}, we have converted these elements of $\hat \Q$ to normal generators of $P_2$.  

\begin{minipage}{0.5\textwidth}
\begin{center}
\includegraphics[scale=1]{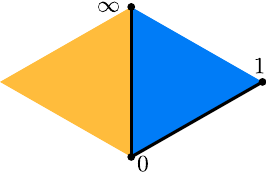}
\end{center}
\end{minipage}
\begin{minipage}{0.5\textwidth}
\begin{center}
\begin{tabular}{ll}
\toprule
Vertex & Generator of $P_2$ \\
\midrule
$\infty$ & $a^2$ \\
$0$ & $b^2$ \\
$1$ & $(ab)^2$\\
\bottomrule
\end{tabular}
\end{center}
\end{minipage}

\begin{prop}
The quotient $F_2/P_2$ is isomorphic to the Klein four-group.
\end{prop}
\begin{proof}
Since all elements of the Klein four-group $K=\langle a,b~|~a^2,b^2,[a,b]\rangle$
have order two, $K$ is a quotient of $F_2/P_2$. Therefore, it suffices to prove the defining relations hold in $F_2/P_2$.
Clearly $a^2=b^2=1$ in $F_2/P_2$ since $a$ and $b$ are primitive in $F_2$. Thus, $a=a^{-1}$ and $b=b^{-1}$. It follows
that $[a,b]=(ab)^2=1$ since $ab$ is primitive in $F_2$.
\end{proof}
%\begin{proof}
%We have 
%$$F_2/P_2=\langle a,b~|~a^2,b^2,(ab)^2\rangle.$$
%Clearly $a$ and $b$ have order two in this quotient. We have
%$[a,b]=(ab)^2=1$ since $a=a^{-1}$ and $b=b^{-1}$. 
%\end{proof}

\subsection*{The case \texorpdfstring{$k=3$}{k=3}}
The triangulated sphere ${\mathcal F}_3$ is a tetrahedron. Below we depict a tree $T$ in an unfolded copy of the tetrahedron.
We have lifted $T$ to a tree $\tilde T$ in the Farey triangulation and labeled the vertices of $T$ by their lifts as elements of $\hat \Q$.
Following Theorem \ref{thm:generators} and Corollary \ref{cor:generators}, we have converted these elements of $\hat \Q$ to normal generators of $P_3$.  

\begin{minipage}{0.5\textwidth}
\begin{center}
\includegraphics[scale=1]{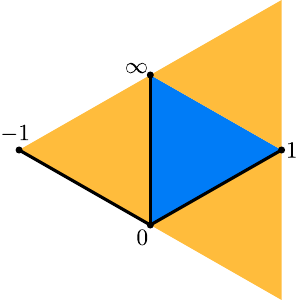}
\end{center}
\end{minipage}
\begin{minipage}{0.5\textwidth}
\begin{center}
\begin{tabular}{ll}
\toprule
Vertex & Generator of $P_3$ \\
\midrule
$\infty$ & $a^3$ \\
$0$ & $b^3$ \\
$1$ & $(ab)^3$\\
$-1$ & $(a b^{-1})^3$\\
\bottomrule
\end{tabular}
\end{center}
\end{minipage}
%
%\begin{figure}
%\includegraphics[scale=.2]{CubesCover.png}
%\caption{Cover corresponding to the $H(\Z)$.}
%\label{fig:H3}
%\end{figure}

\begin{prop}
The quotient $F_2/P_3$ is isomorphic to $H(\Z/3\Z)$.
%, the group of upper triangular matrices with entries in $\Z/3\Z$ and with ones along the diagonal.  
\end{prop}
\begin{proof}
In $H(\Z/3\Z)$, all elements have order three. Thus, $H(\Z/3\Z)$ is a quotient of $F_2/P_3$ and so it suffices to prove
that relations defining $H(\Z/3\Z)$ are satisfied in $F_2/P_3$. We work with the presentation
$$H(\Z) = H(\Z/3\Z)=\left\langle a, b~\Big|~ a^3, b^3, [a,b]^3, \big[a,[a,b]\big], \big[b,[a,b]\big]\right\rangle.$$
Since $a$ and $b$ are primitive, we have $a^3=b^3=1$ in $F_2/P_3$. Also we have
\begin{eqnarray*}
\big[a, [a, b]\big] &=& a^{-1} (b^{-1} a^{-1} b a) a (a^{-1} b^{-1} a b) \\
&=& a^{-1} b^{-1} a^{-1} b a b^{-1} a b \\
&=& (a^{-1} b^{-1})^2 b^2 a b^{-1} a b.
\end{eqnarray*}
Since $a^{-1} b^{-1}$ is primitive in $F^2$, we have $(a^{-1} b^{-1})^3=1$ and thus continuing,
\begin{eqnarray*}
\big[a, [a, b]\big]
&=& b a b^2 a b^{-1} a b.\\
&=& b a b^{-1} a b^{-1} a b.\\
&=& b (a b^{-1})^3 b a^{-1} a b\\
&=&  b^{2} a^{-1} a b \\
&=& 1.
\end{eqnarray*}
Further, since $P_3$ is characteristic, we get $[b, [a,b]] = 1$.
It follows that $[a,b]$ is central, thus $[a,b]^3 = [a^3, b] = 1$ via commutator identities, completing the proof.
\end{proof}

\subsection*{The case \texorpdfstring{$k=4$}{k=4}}
The triangulated sphere ${\mathcal F}_4$ is an octahedron. Below we depict a tree $T$ in an unfolded copy of the octahedron.
We have lifted $T$ to a tree $\tilde T$ in the Farey triangulation and labeled the vertices of $T$ by their lifts as elements of $\hat \Q$.
Following Theorem \ref{thm:generators} and Corollary \ref{cor:generators}, we have converted these elements of $\hat \Q$ to normal generators of $P_4$.  

\begin{minipage}{0.5\textwidth}
\begin{center}
\includegraphics[scale=.8]{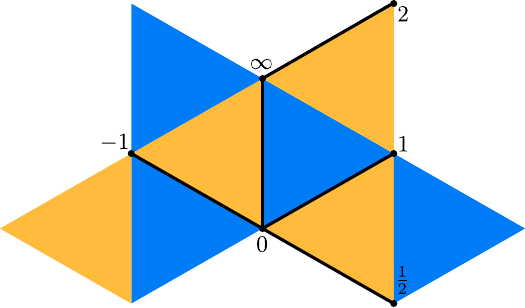}
\end{center}
\end{minipage}
\begin{minipage}{0.6\textwidth}
\begin{center}
\begin{tabular}{ll}
\toprule
Vertex & Generator of $P_4$ \\
\midrule
$\infty$ & $a^4$ \\
$0$ & $b^4$ \\
$1$ & $(ab)^4$\\
$-1$ & $(ab^{-1})^4$\\
$2$ & $(a^2 b)^4$\\
$\frac{1}{2}$ & $(a b^2)^4$\\
\bottomrule
\end{tabular}
\end{center}
\label{table:P4}
\end{minipage}

\subsection*{The case \texorpdfstring{$k=5$}{k=5}}
The triangulated sphere ${\mathcal F}_5$ is an icosahedron. Below we depict a tree $T$ in an unfolded copy of the icosahedron.
We have lifted $T$ to a tree $\tilde T$ in the Farey triangulation and labeled the vertices of $T$ by their lifts as elements of $\hat \Q$.
Following Theorem \ref{thm:generators} and Corollary \ref{cor:generators}, we have converted these elements of $\hat \Q$ to normal generators of $P_5$.  

\begin{center}
\includegraphics[scale=.8]{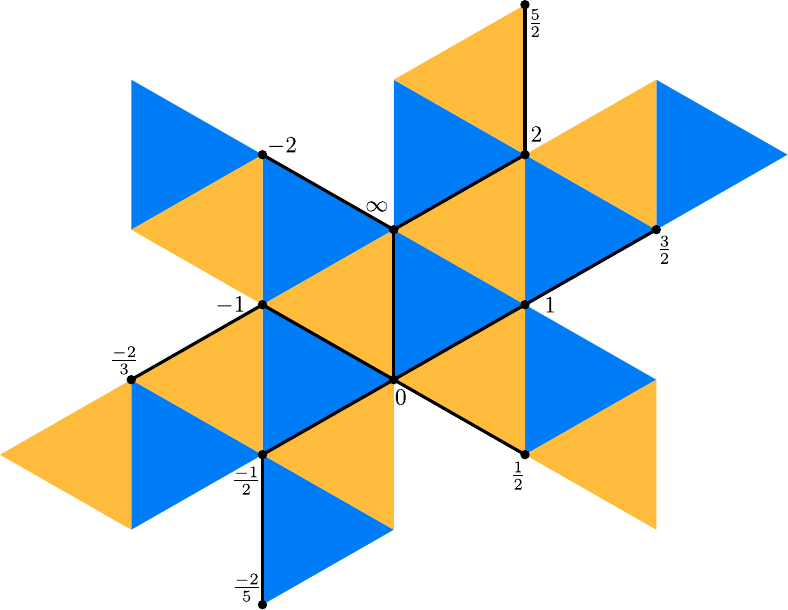}
\end{center}

\begin{center}
\begin{tabular}{ll}
\toprule
Vertex & Generator of $P_5$ \\
\midrule
$\infty$ & $a^5$ \\
$0$ & $b^5$ \\
$1$ & $(ab)^5$\\
$-1$ & $(a b^{-1})^5$\\
$2$ & $(a^2b)^5$\\
$\frac{1}{2}$ & $(ab^2)^5$\\
$-2$ & $(a^{2}b^{-1})^5$\\
$\frac{-1}{2}$ & $(a b^{-2})^5$\\
$\frac{3}{2}$ & $(a^2bab)^5$\\
$\frac{-2}{3}$ & $(a b^{-1} a b^{-2})^5$\\
$\frac{5}{2}$ & $(a^3ba^2b)^5$\\
$\frac{-2}{5}$ & $(a b^{-2} a b^{-3})^5$\\
\bottomrule
\end{tabular}
\end{center}
\label{table:P5}

%% file: representations.tex
\section{Characteristic representations}
\label{sect:representations}

\subsection{Definition and a criterion}

We say that a homomorphism $\rho:F_2 \to \GL(n,\C)$ is a {\em characteristic representation} if for any $\psi \in \Aut(F_2)$ there is a $\Psi \in \Aut\big(\GL(n,\C)\big)$ such that
\begin{equation}
\label{eq:characteristic representation}
\Psi \circ \rho \circ \psi^{-1} (g) = \rho (g) \quad \text{for all $g \in F_2$}.
\end{equation}
The following should be clear:

\begin{prop}
\label{prop:characteristic kernel}
The kernel of a characteristic representation is a characteristic subgroup of $F_2$.
\end{prop}

Recall from \S \ref{sect:primitive} that $\Aut(F_2)=\Aut_+(F_2) \cup \Aut_-(F_2)$.
Our automorphisms of $\GL(n,\C)$ will have one of two forms corresponding to this partition. If $M \in \GL(n,\C)$ then we define 
\begin{equation}
\Psi_M,\bPsi_M \in \Aut\big(\GL(n,\C)\big) 
\quad \text{by} \quad
\Psi_M(X)=MXM^{-1}
\quad \text{and} \quad
\bPsi_M(X)=M\overline{X}M^{-1}.
\end{equation}
We call the map $\bPsi_M$ a {\em conjugate inner automophism}.
%The group of inner automorphisms is $\Inn\big(\GL(n,\C)\big)=\{\Psi_M:~M \in \GL(n,\C\}$
%and the collection of {\em conjugate-inner automorphisms} is $\bInn\big(\GL(n,\C)\big)=\{\bPsi_M:~M \in \GL(n,\C)\}$.
%The union
%$\Inn\big(\GL(n,\C)\big) \cup \bInn\big(\GL(n,\C)\big)$ is a subgroup of $\Aut\big(\GL(n,\C)\big)$.

We say $\rho:F_2 \to \GL(n,\C)$ is an {\em oriented characteristic representation} if the following two statements hold:
\begin{itemize}
\item[(+)] For each $\psi \in \Aut_+(F_2)$ there is an $M \in \GL(n,\C)$
so that \eqref{eq:characteristic representation} holds with $\Psi=\Psi_M$.
\item[(--)] For each $\psi \in \Aut_-(F_2)$ there is an $M \in \GL(n,\C)$
so that \eqref{eq:characteristic representation} holds with $\Psi=\bPsi_M$.
\end{itemize}
%Observe that oriented characteristic representations are examples of characteristic representations.
We will be working exclusively with oriented characteristic representations. 

Based on properties of the tensor product, it can be observed:
\begin{prop}
\label{prop:basic operation}
If $\rho_1:F_2 \to \GL(n_1,\C)$ and $\rho_2:F_2 \to \GL(n_2,\C)$ are oriented characteristic representations then so is their tensor product $\rho_1 \otimes \rho_2:F_2 \to \GL(n_1 n_2,\C)$ and so is the complex-conjugate representation $\overline{\rho_1}$.
\end{prop}

We will now give an elementary method to prove that a homomorphism $\rho$ is an oriented characteristic representation. We single out elements $\psi_1, \psi_2 \in \Aut_+(F_2)$ and $\psi_- \in \Aut_-(F_2)$ whose images in $\Out(F_2)$ generate:
\begin{equation}
\label{eq:generators for Out}
\psi_1(a)=b,~\psi_1(b)=a^{-1}; \quad \psi_2(a)=a,~\psi_2(b)=ab; 
\quad \psi_-(a)=a^{-1},~\psi_-(b)=b.
\end{equation}

We have the following criterion for checking if a representation is oriented characteristic:
\begin{prop}
\label{prop:characteristic criterion}
Let $\rho:F_2 \to \GL(n,\C)$ be a homomorphism. Then $\rho$ is an oriented characteristic representation if and only if the following statements are satisfied:
\begin{enumerate}
\item[(1)] There is an $M_1 \in \GL(n,\C)$ such that $M_1 =\rho(a) M_1 \rho(b)$
and $M_1 \rho(a) = \rho(b) M_1$.
\item[(2)] There is an $M_2 \in \GL(n,\C)$ such that $M_2 \rho(a)=\rho(a) M_2$ and
$M_2 \rho(b) = \rho(ab) M_2.$
\item[(--)] There is an $M_- \in \GL(n,\C)$ such that $M_-=\rho(a) M_- \overline{\rho(a)}$ and
$M_- \overline{\rho(b)} = \rho(b) M_-.$
\end{enumerate}
\end{prop}
We remark that the equations in the respective statements above are simple algebraic manipulations 
of \eqref{eq:characteristic representation} in the special cases where
$(\psi,\Psi)$ is taken to be one the pairs $(\psi_1, \Psi_{M_1})$, $(\psi_2, \Psi_{M_2})$ or
$(\psi_-, \bPsi_{M_-})$ and $g$ is restricted to a pair of generators of $F_2$. (For (1) and (-) we use generators $a$ and $b$, while in (2) we use $a$ and $ab$.)
Thus the ``only if'' direction is clear.
\begin{proof}[Proof of ``if'' direction]
Assume statements (1) (2) and (--) of the proposition hold. We must prove statements (+) and (--) of the definition of oriented characteristic definition. Let 
\begin{equation}
    \label{eq:Delta}
\begin{array}{ll}
\Delta_+=\big\{(M,\psi) \in \GL(n,\C) \times \Aut_+(F_2):~
\text{\eqref{eq:characteristic representation} holds with $\Psi=\Psi_M$}\big\} 
& \text{and} \\
\Delta_-=\big\{(M,\psi) \in \GL(n,\C) \times \Aut_-(F_2):~
\text{\eqref{eq:characteristic representation} holds with $\Psi=\bPsi_M$}\big\}.
\end{array}
\end{equation}
Observe that $\Delta = \Delta_+ \sqcup \Delta_-$ is a group, though the group operation needs adjustment. If $(M',\psi') \in \Delta_s$ with $s \in \{+,-\}$ we define
$$
(M,\psi) \cdot (M',\psi')=
\begin{cases}
(M M', \psi \circ \psi') \in \Delta_{s} & \text{if $(M,\psi) \in \Delta_+$,} \\
(M \overline{M'}, \psi \circ \psi') \in \Delta_{-s} & \text{if $(M,\psi) \in \Delta_-$.}
\end{cases}
$$
(This choice is made to be compatible with composition of inner automorphisms and conjugate inner automorphisms.)
%If statement \eqref{eq:characteristic representation} holds for some collection of pairs in $\Delta_+ \sqcup \Delta_-$ then it also holds for the generated subgroup. 
We must prove that the projection of $\Delta$ to $\Aut(F_2)$ is surjective.

First consider the inner automorphisms of $F_2$, which have the form $\psi_h(g)=hgh^{-1}$ for some $h \in F_2$. By manipulating \eqref{eq:characteristic representation} it can be observed that $\big(\rho(h), \psi_h\big) \in \Delta_+$ for all $h$. 

Now consider $\psi_1$ and $\psi_2$. Observe that \eqref{eq:characteristic representation} holds for all $g \in F_2$ if and only if it holds for a set of generators of $F_2$. As indicated above this proof,
by manipulating \eqref{eq:characteristic representation} in each case, it follows that $(M_1,\psi_1), (M_2,\psi_2) \in \Delta_+$. The elements $\psi_1$ and $\psi_2$ together with the inner automorphisms generate $\Aut_+(F_2)$, so $\Aut_+(F_2)$ is in the image of the projection of $\Delta_+$. 

Similarly consider $\psi_-$. Again by considering \eqref{eq:characteristic representation} in this case,
we see that $(M_-,\psi_-) \in \Delta_-$. The collection $\{\psi_-\} \sqcup \Aut_+(F_2)$ generates $\Aut(F_2)$, so it must be that $\Aut(F_2)$ is in the image of the projection of $\Delta$ as desired.
\end{proof}

\begin{rem}[Orientation reversing elements]
For the main goals of the paper, it would suffice to work with $\Aut_+(F_2)$ rather than all of $\Aut(F_2)$, since $\Aut_+(F_2)$ already acts transitively on primitive elements of $F_2$, and one could define a notion of oriented characteristic representation omitting (--) from the definition. However, all the representations we found have this extra symmetry, and our algorithm for ``improvement'' of representations described in \S \ref{sect:improving} respects this additional symmetry. So, we have opted to consider $\Aut_-(F_2)$ throughout this section for aesthetic reasons at the cost of some minor increase in the complexity of some of our arguments.
\end{rem}

\subsection{Some characteristic representations with finite image}
\label{sect:finite image}
We will now give some finite oriented characteristic representations.

We define $\rho_2:F_2 \to \GL(3,\C)$ by 
\begin{equation}
\label{eq:rho_2}
\rho_2(a)=\text{diag}(-1,-1,1) \quad \text{and} \quad
\rho_2(b)=\text{diag}(1,-1,-1).
\end{equation}
\begin{prop}
\label{prop:rho 2}
The image $\rho_2(F_2)$ is isomorphic to the Klein four group, $C_2 \times C_2$.
The representation $\rho_2$ is oriented characteristic.
\end{prop}
\begin{proof}
The image $\rho(F_2)$ can easily be seen to consist of four elements:
$\rho_2(a)$, $\rho_2(b)$, the identity and $\rho(ab)=\text{diag}(-1,1,-1)$. By inspection the image
is isomorphic to the Klein four group.
By an elementary calculation it can be observed that the statements of Proposition \ref{prop:characteristic criterion} are satisfied when the choices of 
$$M_1=\left(\begin{array}{rrr}
0 & 0 & 1 \\
0 & 1 & 0 \\
1 & 0 & 0
\end{array}\right)
\quad \text{and} \quad
M_2=\left(\begin{array}{rrr}
0 & 1 & 0 \\
1 & 0 & 0 \\
0 & 0 & 1
\end{array}\right)$$
and $M_-=I$ are made.
\end{proof}

For odd numbers $k \geq 3$ define $\rho_k:F_2 \to \GL(k,\C)$ by 
\begin{equation}
\label{eq:rho_k}
\rho_k(a)=\text{diag}(1, \omega, \omega^2, \ldots, \omega^{k-1})
\quad \text{and} \quad
\rho_k(b)=\left(\begin{array}{rrrrrr}
0 & 1 & 0 & 0 & 0 & \ldots \\
0 & 0 & 1 & 0 & 0 & \ldots \\
0 & 0 & 0 & 1  & 0 & \ldots \\
\vdots & \vdots & \vdots & & \ddots  \\
0 & 0 & 0 & 0 & & 1 \\
1 & 0 & 0 & 0 & \dots & 0
\end{array}\right),
\end{equation}
where $\omega=e^{\frac{2 \pi i}{k}}.$ Here $\rho_k(b)$ is a permutation matrix of order $k$.

\begin{prop}
\label{prop:rho odd}
The image $\rho_k(F_2)$ is isomorphic to the Heisenberg group $H(\Z/k\Z)$.
The representation is oriented characteristic: It satisfies the hypotheses of Proposition \ref{prop:characteristic criterion} with the matrix $M_1$ given by
$$(M_{1})_{i,j}=\omega^{(i-1)(j-1)} \quad \text{for $i,j \in \{1,\ldots,k\}$},$$
with $M_2$ given by the diagonal matrix with entries $(M_{2})_{i,i}=\omega^{\frac{-(i-1)(i-2)}{2}}$
and with $M_-=I$.
\end{prop}
\begin{proof}
To see the image is the Heisenberg group, recall that
$$H(\Z/k\Z)=\left\langle a, b~\Big|~ a^k, b^k, [a,b]^k, \big[a,[a,b]\big], \big[b,[a,b]\big]\right\rangle.$$
First we will check that $\rho_k$ factors through $H(\Z/k\Z)$.
It should be clear that $a^k$ and $b^k$ lie in $\ker \rho_k$. By computation we see $\rho_k([a,b])=\omega^{-1} I$. Thus $[a,b]$ is central in the image and $[a,b]^k \in \ker \rho_k$. This shows that the image $\rho_k(F_2)$ is isomorphic to a quotient of $H(\Z/k\Z)$. 
The image must be isomorphic to $H(\Z/k\Z)$ because the homomorphism restricts to an isomorphism of the center of $H(\Z/k\Z)$.

The statements of Proposition \ref{prop:characteristic criterion} for the matrices $M_1$, $M_2$ and $M_-$ listed can be verified by a direct computation. (Calculation carried out by hand, and checked for various values of $k$ with SageMath \cite{sagemath}.)
%Consider (1). Observe the entries of $\rho_k(a) M_1$ and $\rho_k(a) M_1 \rho_k(b)$ are given by 
%$$\big(\rho_k(a) M_1\big)_{i,j}=\omega^{(i-1)j}
%\quad \text{and} \quad 
%\big(\rho_k(a) M_1\rho_k(b)\big)_{i,j}=\omega^{(i-1)(j-1)}.$$
%The entries of $M_1 \rho(a)$ and $\rho(b) M_1$ can both be computed to be
%$$\big(M_1 \rho_k(a)\big)_{i,j}=\omega^{i(j-1)}=\big(\rho_k(b)M_1\big)_{i,j}.$$
%The matrices $M_2$ and $\rho(a)$ are both diagonal and so they commute. The non-zero entries of $\rho(ab)$ are given below with indices written modulo $k$ as are the non-zero entries of $M_2 \rho_k(b)$ and $\rho_k(ab)M_2$:
%$$\rho_k(ab)_{i,i+1}=\omega^{i-1} \quad \text{and} \quad
%\big(M_2 \rho_k(b)\big)_{i,i+1}=\omega^{\frac{-(i-1)(i-2)}{2}}=\big(\rho_k(ab)M_2\big)_{i,i+1}.$$
%The computation checking the statement (--) should be clear since $M_-=I$. 
\end{proof}

Observe that the images of $\rho_k$ are matrices with entries in $\Z[\omega]$. 
Later we will need the following observation:

\begin{prop}
\label{prop:matrix group}
Fix an odd $k \geq 3$. Let $M_{k,k}$ denote the additive group of $k \times k$ matrices with entries in $\Z[\omega]$. The subgroup of $M_{k,k}$ generated by $\{\rho_k(g):~g \in F_2\}$ has finite index.
\end{prop}
\begin{proof}
Let $E_{i,j}$ denote the matrix with a $1$ in the entry in row $i$ and column $j$ but with all other entries equal to zero. It suffices to show that $k \omega^n E_{i,j}$ is in the generated subgroup for all $i,j \in \{1, \ldots, k\}$ and all $n \in \{0, \ldots,k-1\}$. By direct computation we observe
$$k E_{1,1}=\sum_{\ell=0}^{k-1} \rho_k(a^\ell).$$
Utilizing the action of $\rho_k(b)$ as a permutation matrix we can then see
$$E_{i,j}=\rho_k(b^{1-i}) \cdot E_{1,1} \cdot \rho_k(b^{j-1}).$$
Thus, $k E_{i,j}$ is in this generated subgroup as well. Finally to get the powers of $\omega$ observe that $\rho_k([b,a])=\omega I$.
\end{proof}

\begin{cor}
\label{cor:irreducible}
For odd $k \geq 3$, the representation $\rho_k$ is irreducible.
\end{cor}
\begin{proof}
Any subspace of $\C^k$ which is invariant under $\rho_k$ must be mapped into itself
by all elements of  the subgroup of $M_{k,k}$ generated by $\{\rho_k(g):~g \in F_2\}$.
The previous proposition implies that there is no such non-zero proper subspace.
\end{proof}

\subsection{Improving characteristic representations}
\label{sect:improving}
We will now  explain a process which can take an oriented characteristic representation $\rho:F_2 \to \GL(n,\C)$
and produce a new oriented characteristic representation $\tilde \rho:F_2 \to \GL(\tilde n,\C)$ where $\tilde n \geq n$ and hopefully the $\ker \tilde \rho$ is strictly smaller than $\ker \rho$. 

Fix $\rho$ for this subsection.
We will consider deformations of $\rho$ into the affine group $\Aff(n)=\C^n \rtimes \GL(n,\C)$ where the product in $\Aff(n)$ is given by
\begin{equation}
\label{eq:product rule}
(\mathbf{v}, M) \cdot (\mathbf{w},N)=(\mathbf{v}+M\mathbf{w},MN).
\end{equation}
The group $\GL(n,\C)$ is isomorphic to a subgroup $\Aff(n)$ via the map $M \mapsto (\0, M)$, and this explains how to multiply elements of $\GL(n,\C)$ and $\Aff(n)$.
Let $\pi_1: \Aff(n) \to \C^n$ and $\pi_2: \Aff(n) \to \GL(n,\C)$ be the natural projections (noting that $\pi_1$ is not a homomorphism).
We will say that an {\em affable} representation 
$\hat \rho: F_2 \to \Aff(n)$ is a homomorphism for which $\pi_2 \circ \hat \rho=\rho.$
We use $\CA$ to denote the collection of all affable representations. Observe:

\begin{prop}
\label{prop:linear}
The collection $\CA$ is a vector space over $\C$ when endowed with the operations of addition and scalar multiplication defined by
$$(\hat \rho_1+\hat \rho_2)(g)=\big(\pi_1 \circ \hat \rho_1(g)+\pi_1 \circ \hat \rho_2(g),\rho(g)\big) \quad \text{and} $$
$$(\lambda \hat \rho_1)(g)=\big(\lambda \pi_1 \circ \hat \rho_1(g), \rho(g)\big)$$
for all $\hat \rho_1, \hat \rho_2 \in \CA$, all $\lambda \in \C$ and all $g \in F_2$. In particular,
for any $g$ the map $\eval_g:\CA \to \C^n$ defined by $\eval_g(\hat \rho)=\pi_1 \circ \hat \rho(g)$ is linear.
\end{prop}
\begin{proof}[Discussion of proof]
The operations are clearly linear in nature, but it must be checked that $\hat \rho_1+\hat \rho_2$ and $\lambda \hat \rho_1$ define group homomorphisms (assuming $\hat \rho_1$ and $\hat \rho_2$ are group homomorphisms). We leave this elementary check to the reader.
\end{proof}

\begin{prop}
\label{prop:coordinates}
Recall $a$ and $b$ denote the generators of $F_2$. The map $\eval_a \times \eval_b:\CA \to \C^n \times \C^n$ is a vector space isomorphism.
\end{prop}
\begin{proof}
It should be clear that this defines a homomorphism between vector spaces by definition of the operations
in Proposition \ref{prop:linear}. It is an isomorphism because the images of the generators determine the homomorphism; the inverse map sends $(\va, \vb)$ to the homomorphism determined by the following images of the generators of $F_2$:
\begin{equation}
\label{eq:a,b}
a\mapsto \big(\va, \rho(a)\big)
\quad \text{and} \quad
b \mapsto \big(\vb, \rho(b)\big).
\end{equation}
\end{proof}

Let $\Conj:\C^n \times \CA \to \CA$ be the action defined by post-conjugation by $\C^n \subset \Aff(n)$:
\begin{equation}
\label{eq:conj}
\Conj_\vv(\hat \rho)(g)=(\vv,I) \cdot \hat \rho(g) \cdot (-\vv,I) \quad \text{for all $g \in F_2$},
\end{equation}
where $I$ denotes the identity element of $\GL(n,\C)$. When $\CA$ is viewed as isomorphic to $\C^{2n}$,
we see that each $\Conj_\vv$ acts by translation on $\CA$ (i.e., $\Conj_\vv(\hat \rho)-\hat \rho$ does not depend on $\hat \rho$):
\begin{prop}
\label{prop:conjugation}
For each $\vv \in \C^n$, each $\hat \rho \in \CA$ and each $g \in F_2$ we have
\begin{equation}
\label{eq:translation vector}
\big(\Conj_\vv(\hat \rho)-\hat \rho\big)(g)= \Big(\big(I-\rho(g)\big)\vv, \rho(g)\Big).
\end{equation}
\end{prop}
We call $\Conj_\vv(\hat \rho)-\hat \rho$ the {\em translation vector} of $\Conj_\vv$.
\begin{proof}
This follows from the computation in $\Aff(n)$:
$$\Conj_\vv(\hat \rho)(g)=(\vv,I) \cdot \big(\pi_1 \circ \hat \rho(g), \rho(g)\big) \cdot (-\vv,I)=
\big(\vv+\pi_1 \circ \hat \rho(g)+\rho(g)(-\vv), \rho(g)\big).
$$
\end{proof}

Let $\sim$ denote the equivalence relation on $\CA$ where 
\begin{equation}
\label{eq:sim}
\hat \rho_1 \sim \hat \rho_2 \quad \text{if there is a $\vv \in \C^n$ satisfying $\Conj_\vv(\hat \rho_1) =\hat \rho_2$.}
\end{equation}

\begin{cor}
The quotient $\CA/\sim$ is a complex vector space with operations induced by those of $\CA$. 
\end{cor}
\begin{proof}
It needs to be observed that the operations of addition and scalar multiplication induce well defined actions on $\CA/\sim$. This follows from linearity of the translation vector of \eqref{eq:translation vector} in $\vv \in \C^n$. 
\end{proof}

Recall that $\rho$ is a fixed homomorphism. Recall the definition of $\Delta=\Delta_+ \sqcup \Delta_-$ in \eqref{eq:Delta} from the proof of Proposition \ref{prop:characteristic criterion} and recall that
$\rho$ is an oriented characteristic representation if and only if the projection of $\Delta = \Delta_+ \sqcup \Delta_-$ to $\Aut(F_2)$ is surjective.

We view $\GL(n,\C)$ as a subgroup of $\Aff(n)$. Conjugation by an element of $\GL(n,\C)$ induces an automorphism of $\Aff(n)$. 

We use $\GL(\CA)$ to denote the group of linear automorphisms of $\CA$ and $\overline{\GL}(\CA)$ to denote the collection of conjugate-linear automorphisms. 
Together, $\GL(\CA) \cup \overline{\GL}(\CA)$ forms a group.
We have the following:
\begin{lemma}
\label{lem:N}
There is a homomorphism $N: \Delta \to \GL(\CA) \cup \overline{\GL}(\CA)$ such that 
\begin{enumerate}
\item[(+)] If $(M,\psi) \in \Delta_+$ and $\hat \rho \in \CA$ then $N_{M,\psi} \in \GL(\CA)$ and 
$$N_{M,\psi}(\hat \rho)(g)=M \cdot \big(\hat \rho \circ \psi^{-1}(g)\big) \cdot M^{-1} \quad \text{for all $g \in F_2$}.$$
\item[(--)] If $(M,\psi) \in \Delta_-$ and $\hat \rho \in \CA$ then $N_{M,\psi} \in \overline{\GL}(\CA)$ and
$$N_{M,\psi}(\hat \rho)(g)=M \cdot \big(\overline{\hat \rho \circ \psi^{-1}(g)}\big) \cdot M^{-1} \quad \text{for all $g \in F_2$}.$$
\end{enumerate}
Each $N_{M,\psi}$ sends $\sim$-equivalence classes to $\sim$-equivalence classes and so induces an automorphism $N^\sim_{M,\psi} \in \GL(\CA/\sim) \cup \overline{\GL}(\CA/\sim)$. Furthermore, the induced map 
$$N^\sim:\Delta \to \GL(\CA/\sim) \cup \overline{\GL}(\CA/\sim) \quad \text{given by} \quad (M,\psi) \mapsto N^\sim_{M,\psi}$$ 
is a homomorphism. 
\end{lemma}
\begin{proof}
Since $M \in \GL(n,\C)$ and $\psi \in \Aut(F_2)$, it should be clear that the definitions provided for $N_{M,\psi}(\hat \rho)$ give a homomorphism $F_2 \to \Aff(n)$.
Writing $\hat \rho(g)=\big(\pi_1 \circ \hat \rho(g),\rho(g)\big)$ (using affability of $\hat \rho$) we see
that when $(M,\psi) \in \Delta_+$ we have
\begin{equation}
\label{eq:N action}
\begin{array}{rcl}
N_{M,\psi}(\hat \rho)(g) & = & \Big(M \cdot \pi_1 \circ \hat \rho \circ \psi^{-1}(g), M \cdot \big(\rho \circ \psi^{-1}(g)\big) \cdot M^{-1}\Big) \\
& = & \big(M \cdot \pi_1 \circ \hat \rho \circ \psi^{-1}(g), \rho(g)\big)
\end{array}
\end{equation}
with the last step given by definition of $\Delta_+$ in \eqref{eq:Delta}. To see linearity observe that $\pi_1 \circ \hat \rho \circ \psi^{-1}(g)$ varies linearly in $\hat \rho$ by Proposition \ref{prop:linear} and we are simply postcomposing with the linear action of $M \in \GL(n, \C)$. 
Similarly if $(M,\psi) \in \Delta_-$,
\begin{equation}
\label{eq:N action-}
\begin{array}{rcl}
N_{M,\psi}(\hat \rho)(g) & = & \big(M \cdot \overline{\pi_1 \circ \hat \rho \circ \psi^{-1}(g)}, M \cdot\overline{\rho \circ \psi^{-1}(g)} \cdot M^{-1}\big) \\
& = & \big(M \cdot \overline{\pi_1 \circ \hat \rho \circ \psi^{-1}(g)}, \rho(g)\big).
\end{array}
\end{equation}
Observe that $N_{M,\psi}$ is conjugate-linear in this case.

Now we must check that the linear action respects $\sim$ equivalence classes. Suppose $\hat \rho_1 \sim \hat \rho_2$.
By Proposition \ref{prop:conjugation}, this is true if and only if there is a $\vv \in \C^n$ such that
\begin{equation}
\label{eq:difference}
(\hat \rho_1-\hat \rho_2)(g)=\Big(\big(I-\rho(g)\big)\vv, \rho(g)\Big) \quad \text{for all $g \in F_2$}.
\end{equation}
Fix such a $\vv$ and let $\hat \rho_\vv \in \CA$ be defined as in the right side of \eqref{eq:difference}.
Then by linearity or conjugate-linearity of $N_{M,\psi}$ we have
$$N_{M,\psi}(\hat \rho_1) - N_{M,\psi}(\hat \rho_2)=N_{M,\psi}(\hat \rho_\vv).$$
By \eqref{eq:N action} if $(M,\psi) \in \Delta_+$ we have
$$N_{M,\psi}(\hat \rho_\vv)(g)=\Big(M \cdot \big(I-\rho \circ \psi^{-1}(g)\big)\vv, \rho(g)\Big)=
\Big(\big(I-\rho(g)\big)M\vv, \rho(g)\Big),$$
where we are using the identity $M \cdot \big(\rho \circ \psi^{-1}(g)\big) \cdot M^{-1}=\rho(g)$ again in the second step. Then Proposition \ref{prop:conjugation} tells us that $N_{M,\psi}(\hat \rho_1) \sim N_{M,\psi}(\hat \rho_2)$. Similarly if $(M,\psi) \in \Delta_-$ we have
$$N_{M,\psi}(\hat \rho_\vv)(g)=\Big(M \cdot \overline{\big(I-\rho \circ \psi^{-1}(g) \big)\vv}, \rho(g)\Big)= \Big(\big(I-\rho(g)\big)M \overline{\vv}, \rho(g)\Big),$$
and again Proposition \ref{prop:conjugation} tells us that $N_{M,\psi}(\hat \rho_1) \sim N_{M,\psi}(\hat \rho_2)$.
\end{proof}

It will be useful later to note that inner automorphisms act trivially on $\CA/\sim$:

\begin{prop}
\label{prop:inner automorphisms act trivially}
Let $\psi_h \in \Aut(F_2)$ denote the inner automorphism $g \mapsto h g h^{-1}$. Then for all $z \in \C \smallsetminus \{0\}$, we have
$$(z\rho(h),\psi_h) \in \Delta_+ \quad \text{and} \quad  N^\sim_{z\rho(h),\psi_h}([\hat \rho])=z [\hat \rho] \quad \text{for all $[\hat \rho] \in \CA/\sim$.}$$
In particular, $N^\sim_{\rho(h),\psi_h}$ acts trivially on $\CA/\sim$.
\end{prop}
\begin{proof}
Fix $z \in \C \smallsetminus \{0\}$ and fix $h \in F_2$. Recall that $(z\rho(h),\psi_h) \in \Delta_+$ if and only if 
$$\big(z \rho(h)\big) \cdot \big(\rho \circ \psi_h^{-1}(g)\big) \cdot \big(z \rho(h)\big)^{-1}=\rho(g) \quad \text{for all $g \in G$}.$$
The $z$ and $z^{-1}$ cancel and left side simplifies as
$$\big(z \rho(h)\big) \cdot \big(\rho \circ \psi_h^{-1}(g)\big) \cdot \big(z \rho(h)\big)^{-1} = 
\rho(h) \rho(h^{-1} g h) \rho(h)^{-1}=\rho(g).$$
Now fix any $\hat \rho \in \CA$ and observe $\hat \rho \circ \psi_h^{-1}(g)=\hat \rho(h^{-1}) \hat \rho(g) \hat \rho(h)$.
Choose $\vv,\vw \in \C^n$ satisfying $\hat \rho(g)=\big(\vv,\rho(g)\big)$ and $\hat \rho(h)=\big(\vw,\rho(h)\big)$. Then $\hat \rho(h)^{-1}=(-\rho(h)^{-1} \vw, \rho(h)^{-1})$ and thus
$$\begin{array}{rcl}
\hat \rho \circ \psi_h^{-1}(g) & = & (-\rho(h)^{-1} \vw, \rho(h)^{-1})\cdot \big(\vv,\rho(g)\big) \cdot \big(\vw,\rho(h)\big) \\
& = & \Big(\rho(h)^{-1}\big(\rho(g)-I\big)\vw+\rho(h)^{-1}\vv,\rho(h^{-1}gh)\Big).
\end{array}
$$
By definition,
$$N_{z\rho(h),\psi_h}(\hat \rho)(g)=\big(z\rho(h)\big) \cdot \big(\hat \rho \circ \psi_h^{-1}(g)\big) \cdot \big(z \rho(h)\big)^{-1}.$$
By combining with the above we see
$$\big[N_{z\rho(h),\psi_h}(\hat \rho)-z\hat \rho\big](g)=\Big(\big(I-\rho(g)\big) (-z\vw),\rho(g)\Big),$$
and so by Proposition \ref{prop:conjugation}
$N_{z\rho(h),\psi_h}(\hat \rho) \sim z\hat \rho$. 
\end{proof}

Fix an integer $k \geq 2$. Recall $P_k \subset F_2$ denotes the subgroup generated by the $k$-th powers of primitive elements in $F_2$. 
Assume $P_k \subset \ker \rho$. The collection of {\em $k$-affable representations} is 
\begin{equation}
\label{eq:k-affable}
\CA_k=\{\hat \rho \in \CA~:~ P_k \subset \ker \hat \rho\}.
\end{equation}
As a consequence of Proposition \ref{prop:linear}, $\CA_k$ is a linear subspace of $\CA$:
it is the intersection of the kernels of the linear maps $\eval_{p^k}$ taken over all primitive $p \in F_2$.

We have:
\begin{prop}~
\label{prop:interaction with A k}
\begin{enumerate}
\item Each $\sim$-equivalence class is either contained in or disjoint from $\CA_k$.
\item For each $(M,\psi) \in \Delta$, $\CA_k$ is invariant under $N_{M,\psi}$.
\end{enumerate}
\end{prop}
\begin{proof}
Since $P_k \subset F_2$ is characteristic, if $f:\CA \to \CA$ is such that $\ker f(\hat \rho)$ differs from $\ker \hat \rho$ by an automorphism of $F_2$ for every affable $\hat \rho$, then
$\CA_k$ is invariant under $f$. This holds in the cases of $f$ given by $\Conj_\vv$ and $N_{M,\psi}$, and these cases cover the respective cases of the proposition.
\end{proof}

Summarizing the results above, we see that $\CA_k/\sim$ is a linear subspace of $\CA/\sim$, and $N^\sim_{M,\psi}(\CA_k/\sim)=\CA_k/\sim$ for all 
$(M,\psi) \in \Delta$.

Choose any subspace $\CI \subset \CA_k/\sim$ which is invariant under the action of $N^\sim_{M,\psi}$ for $(M,\psi) \in \Delta$.
Ideally we would take $\CI=\CA_k/\sim$ to get the largest invariant space possible. (Later in the proof of Theorem \ref{thm:k odd}
we do not prove that our choice of $\CI$ is all of $\CA_k/\sim$.)

Let $m= \dim \CI$. Choose $\hat \rho_1, \ldots, \hat \rho_m \in \CA_k$ such that
the images in $\CA_k/\sim$ form a basis for $\CI$. In block matrix form we define
\begin{equation}
\label{eq:tilde rho}
\begin{array}{l}
\tilde \rho:F_2 \to \GL(n+m,\C); \quad g \mapsto
\left(\begin{array}{rr}
\rho(g) & Q(g) \\
0 & I 
\end{array}\right)\in \GL(n+m,\C) \\ 
\text{where} \quad  Q(g)=\left(\begin{array}{rrrr} \pi_1 \circ \hat \rho_1(g) & \pi_1 \circ \hat \rho_2(g) & \ldots & \pi_1 \circ \hat \rho_m(g) \end{array}\right).
\end{array}
\end{equation}
Here each $\pi_1 \circ \hat \rho_i(g)$ is interpreted as the $i$-th column vector of $Q(g)$. Then:

\begin{thm}
\label{thm:tilde rho}
Assume $\rho:F_2 \to \GL(n,\C)$ is an oriented characteristic representation with $P_k \subset \ker \rho$. 
Define $\CI$, $m$, $\hat \rho_1$, \ldots, $\hat \rho_m$ and $\tilde \rho$ as above. 
Then $\tilde \rho$ is also an oriented characteristic representation with 
$P_k \subset \ker \tilde \rho$.
Furthermore, there is a short exact sequence of the form
$$1 \to \tilde \rho(\ker \rho) \to F_2/\ker \tilde \rho \to F_2 / \ker \rho \to 1,$$
and $\tilde \rho(\ker \rho)$ is a torsion-free abelian group. 
\end{thm}

\begin{proof}
First we will show that $\tilde \rho$ is a group homomorphism. Considering the block form of the image, observe that it suffices to understand the top right block (since we are given that $\rho$ is a homomorphism). Checking that $\tilde \rho(g_1 g_2)=\tilde \rho(g_1) \tilde \rho(g_2)$ then reduces to checking that 
$$Q(g_1 g_2)=Q(g_1)+\rho(g_1) Q(g_2).$$
Checking this for column $i$ amounts to checking that 
$$\pi_1 \circ \hat \rho_i(g_1 g_2)=\pi_1 \circ \hat \rho_i(g_1)+\rho(g_1) \cdot \pi_1 \circ \hat \rho_i(g_2)$$
which is true because $\hat \rho_i$ is a homomorphism to $\Aff(n)$ which has product rule as in \eqref{eq:product rule}.

%Observe that $P_k \subset \tilde \rho$ since $P_k \subset \ker \rho$ and each $\hat \rho_i \in \CA_k$. 
From \eqref{eq:tilde rho} and by definition of $\CA_k$ we see that $\tilde \rho(g^k)=I$ for each primitive $g \in F_2$ 
guaranteeing that $P_k \subset \ker \tilde \rho$. 

Exactness of the provided sequence should be clear. The group $\tilde \rho(\ker \rho)$ is torsion-free and abelian because for each $g \in \ker \rho$ we have
$$\tilde \rho(g)=\left(\begin{array}{rr} I & Q(g) \\ 0 & I \end{array}\right).$$
In particular, $\tilde \rho(\ker \rho)$ is an additive subgroup of $\C^{mn}$.

It remains to show that $\tilde \rho$ is an oriented characteristic representation. Choose any $\psi \in \Aut(F_2)$. 
Let $s \in \{+,-\}$ be such that $\psi \in \Aut_s(F_2)$. Define
$$\tilde \rho_0:F_2 \to \GL(n+m,\C) \quad \text{by} \quad
\tilde \rho_0(g)=\begin{cases}
\tilde \rho \circ \psi^{-1}(g) & \quad \text{if $s=+$} \\
\overline{\tilde \rho \circ \psi^{-1}(g)} & \quad \text{if $s=-$}. \\
\end{cases}$$
We need to show that $\tilde \rho_0$ is conjugate by an element of $\GL(m+n,\C)$ to $\tilde \rho$.
We will demonstrate this by applying a sequence of conjugations.

First since $\rho$ is an oriented characteristic representation, there is a matrix
$M \in \GL(n,\R)$ such that $(M,\psi) \in \Delta_s$. This guarantees that either 
\begin{equation}
\label{eq:characteristic2}
M \cdot [\rho \circ \psi^{-1}(g)] \cdot M^{-1}=\rho(g) 
\quad \text{or} \quad
M \cdot \overline{\rho \circ \psi^{-1}(g)} \cdot M^{-1}=\rho(g) 
\end{equation}
for all $g \in F_2$ depending on the sign $s$. Define $\tilde \rho_1$ to be a conjugate of $\tilde \rho_0$ formed as follows:
$$\tilde \rho_1(g)=\left(\begin{array}{rr}
M & 0 \\
0 & I 
\end{array}\right) \cdot \tilde \rho_0(g) \cdot
\left(\begin{array}{rr}
M^{-1} & 0 \\
0 & I 
\end{array}\right).
$$
Recall from Lemma \ref{lem:N} there is an $N_{M,\psi}$ in $\GL(\CA)$ or $\overline{\GL(\CA)}$ (depending on $s$)
describing the conjugation action on $\CA$. Let $\mathrm{Col}_i(X)$ denote the $i$-th column of a matrix $X$. 
By \eqref{eq:characteristic2} and description of $N_{M,\psi}$ we have
$$\tilde \rho_1(g)=\left(\begin{array}{rr}
\rho(g) & Q_1(g) \\
0 & I 
\end{array}\right)
\quad \text{where} \quad
\quad \mathrm{Col}_i\big(Q_1(g)\big)=\pi_1 \big( N_{M,\psi}(\hat \rho_i)(g)\big)$$
for $i \in \{1,\ldots,m\}$. 

Let $\tilde \CI \subset \CA_k$ denote the preimage of $\CI$ under the quotient map $\CA_k \to \CA_k/\sim$; this is the union of the equivalence classes in $\CI$. Recall that $\CI \subset \CA_k/\sim$ is $N^\sim_{M,\psi}$-invariant. In addition $\sim$-equivalence classes are permuted by $N_{M,\psi}$; see Lemma \ref{lem:N}. It follows that $\tilde \CI$ is $N_{M,\psi}$-invariant. We therefore have that $N_{M,\psi}(\hat \rho_i)$ lies in $\tilde \CI$ for $i \in \{1, \ldots, m\}$. Let $\CI_L=\mathrm{span} \{\hat \rho_1, \ldots, \hat \rho_m\}$. 
From our choice of $\hat \rho_1, \ldots, \hat \rho_m$, the space $\CI_L$ is a lift of $\CI$; i.e., the quotient map $\CA_k \to \CA_k/\sim$ restricts to an isomorphism $\CI_L \to \CI$. Recall that the $\sim$-equivalence classes are $\Conj_\vv$ orbits; see \eqref{eq:sim}. Thus for each $i \in \{1,\ldots,m\}$ there is a unique vector $\vv_i$ such that 
$$\Conj_{\vv_i} \circ N_{M,\psi}(\hat \rho_i) \in \CI_L.$$
We now define a homomorphism $\tilde \rho_2:F_2 \to \GL(n+m,\C)$ conjugate to $\tilde \rho_1$ by
$$\tilde \rho_1(g)=\left(\begin{array}{rr}
I & V \\
0 & I 
\end{array}\right) \cdot \tilde \rho_1(g) \cdot
\left(\begin{array}{rr}
I & -V \\
0 & I 
\end{array}\right)
\quad \text{where} \quad
V =\left(\begin{array}{rrr} \vv_1 & \ldots & \vv_m \end{array}\right).
$$
By definition of $\Conj$ in \eqref{eq:conj} we see that 
$$\tilde \rho_2(g)=\left(\begin{array}{rr}
\rho(g) & Q_2(g) \\
0 & I 
\end{array}\right)
\quad \text{where} \quad
\mathrm{Col}_i\big(Q_2(g)\big)=\pi_1 \big(\Conj_{\vv_i} \circ N_{M,\psi}(\hat \rho_i)\big)$$
for all $i \in \{1,\ldots,m\}$.

Let $p$ denote the isomorphism $\CI_L \to \CI$ mentioned above. Since this is an isomorphism, there is an 
$N_L \in \GL(\CI_L) \cup \overline{\GL}(\CI_L)$ such that $p \circ N_L=N^\sim_{M,\psi}|_{\mathcal I} \circ p$. Then in particular we have
$$\Conj_{\vv_i} \circ N_{M,\psi}(\hat \rho_i)=N_L(\hat \rho_i)
\quad \text{for all $i \in \{1,\ldots,m\}$.}$$
As a consequence we see that $\big\{\Conj_{\vv_i} \circ N_{M,\psi}(\hat \rho_i):~i=1,\ldots,m\big\}$ is a basis for $\CI_L$. Thus there is a matrix $R=(R_{i,j}) \in \GL(m,\C)$ such that 
$$\hat \rho_j=\sum_{i=1}^m R_{i,j} \Conj_{\vv_i} \circ N_{M,\psi}(\hat \rho_i)
\quad \text{for all $j \in \{1,\ldots,m\}$.}$$
Then by linearity of the evaluation maps (see Proposition \ref{prop:linear}), for each $g \in F_2$ we have
\begin{equation}
\label{eq:eval instance}
\pi_1 \big(\hat \rho_j(g)\big)=\sum_{i=1}^m R_{i,j} \pi_1\big(\Conj_{\vv_i} \circ N_{M,\psi}(\hat \rho_i)(g)\big)
\quad \text{for all $j \in \{1,\ldots,m\}$.}
\end{equation}
Recall the left side gives the columns of $Q(g)$, which is the top right submatrix of $\tilde \rho(g)$; see \eqref{eq:tilde rho}. Thus this equation expresses column $j$ of $Q(g)$ as a linear combination of columns of $Q_2(g)$ with weights given by entries in the $j$-th column of $R$.
Thus we have that $Q_2(g) \cdot R=Q(g)$ for all $g$.
We define the conjugate $\tilde \rho_3$ of $\tilde \rho_2$ by 
$$\tilde \rho_3(g)=\left(\begin{array}{rr}
I & 0 \\
0 & R^{-1} \\ 
\end{array}\right) \cdot \tilde \rho_2(g) \cdot
\left(\begin{array}{rr}
I & 0 \\
0 & R 
\end{array}\right)=\left(
\begin{array}{rr}
\rho(g) & Q_2(g) R \\
0 & 1
\end{array}\right)=\tilde \rho(g).
$$
Since $\tilde \rho_3=\tilde \rho$ we have produced the desired conjugacy.
\end{proof}

\subsection{Case \texorpdfstring{$k=6$}{k=6}}
We define $\rho_6$ to be the tensor product $\rho_2 \otimes \rho_3$ where
$\rho_2$ and $\rho_3$ are defined as in \eqref{eq:rho_2} and \eqref{eq:rho_k}.
Then $\rho_6$ may be thought of as a homomorphism $F_2 \to \GL(9,\C)$. Letting $\omega=e^{\frac{2 \pi i}{3}}$, we have the formulas:
$$\rho_6(a)=\diag(-1,-\omega,-\omega^2; -1,-\omega,-\omega^2; 1,\omega,\omega^2),$$
$$\rho_6(b)=\left(\begin{array}{rrr|rrr|rrr}
0 & 1 & 0 & 0 & 0 & 0 & 0 & 0 & 0 \\
0 & 0 & 1 & 0 & 0 & 0 & 0 & 0 & 0 \\
1 & 0 & 0 & 0 & 0 & 0 & 0 & 0 & 0 \\
\hline
 0 & 0 & 0 & 0 & -1 & 0 & 0 & 0 & 0 \\
0 & 0 & 0 & 0 & 0 & -1 & 0 & 0 & 0 \\
0 & 0 & 0 & -1 & 0 & 0 & 0 & 0 & 0 \\
\hline
 0 & 0 & 0 & 0 & 0 & 0 & 0 & -1 & 0 \\
0 & 0 & 0 & 0 & 0 & 0 & 0 & 0 & -1 \\
0 & 0 & 0 & 0 & 0 & 0 & -1 & 0 & 0
\end{array}\right).$$

Applying the improvement algorithm of Theorem \ref{thm:tilde rho} to $\rho_6$
can be shown by calculation to give rise to the representation
$\tilde \rho_6:F_2 \to \GL(12,\C)$ defined as block matrices as
$$\tilde \rho_6(a)=\left(\begin{array}{rr}
\rho_6(a) & 0 \\
0 & I \end{array}\right) \quad \text{and}$$
$$\tilde \rho_6(b)=\left(\begin{array}{rr}
\rho_6(b) & B \\
0 & I \end{array}\right)
\quad \text{where} \quad B=\left(\begin{array}{rrr}
1 & 0 & 0 \\
0 & 0 & 0 \\
-1 & 0 & 0 \\
0 & 1 & 0 \\
0 & 0 & 0 \\
0 & -1 & 0 \\
0 & 0 & 0 \\
0 & 0 & 1 \\
0 & 0 & 0
\end{array}\right).$$

We will not present the computational proof that $\tilde \rho_6$ arises from $\rho_6$ by applying 
Theorem \ref{thm:tilde rho} with $\CI=\CA_6$. However we will demonstrate that 
it is an oriented characteristic representation:
\begin{prop}
\label{prop:6}
The homomorphism $\tilde \rho_6$ is an oriented characteristic representation. The kernel $\ker \tilde \rho_6$ contains $P_6$ and is infinite index in $F_2$. Furthermore, there is a short exact sequence of groups of the form
$$1 \to \Z^{18} \to F_2/\ker \tilde \rho_6 \to C_2 \times C_2 \times H(\Z/3\Z) \to 1.$$
\end{prop}
\begin{proof}
To see the representation is oriented characteristic, observe that the criterion of Proposition \ref{prop:characteristic criterion} is satisfied with the choices of matrices $M_-=I$ and $M_1$ and $M_2$ as below:
$$M_1=\left(\begin{array}{rrrrrrrrr|rrr}
0 & 0 & 0 & 0 & 0 & 0 & 1 & 1 & 1 & 0 & 0 & -\frac{1}{2} \\
0 & 0 & 0 & 0 & 0 & 0 & 1 & \omega & \omega^2 & 0 & 0 & 1 \\
0 & 0 & 0 & 0 & 0 & 0 & 1 & \omega^2 & \omega & 0 & 0 & 1 \\
0 & 0 & 0 & 1 & 1 & 1 & 0 & 0 & 0 & 0 & 0 & 0 \\
0 & 0 & 0 & 1 & \omega & \omega^2 & 0 & 0 & 0 & 0 & -2 \omega - 1 & 0 \\
0 & 0 & 0 & 1 & \omega^2 & \omega & 0 & 0 & 0 & 0 & 2 \omega + 1 & 0 \\
1 & 1 & 1 & 0 & 0 & 0 & 0 & 0 & 0 & 1 & 0 & 0 \\
1 & \omega & \omega^2 & 0 & 0 & 0 & 0 & 0 & 0 & -1 & 0 & 0 \\
1 & \omega^2 & \omega & 0 & 0 & 0 & 0 & 0 & 0 & -1 & 0 & 0 \\
\hline
0 & 0 & 0 & 0 & 0 & 0 & 0 & 0 & 0 & 0 & 0 & -\frac{3}{2} \\
0 & 0 & 0 & 0 & 0 & 0 & 0 & 0 & 0 & 0 & -2 \omega - 1 & 0 \\
0 & 0 & 0 & 0 & 0 & 0 & 0 & 0 & 0 & -2 & 0 & 0
\end{array}\right).$$
$$M_2=\left(\begin{array}{rrrrrrrrr|rrr}
0 & 0 & 0 & 1 & 0 & 0 & 0 & 0 & 0 & 0 & 0 & 0 \\
0 & 0 & 0 & 0 & 1 & 0 & 0 & 0 & 0 & 0 & 0 & 0 \\
0 & 0 & 0 & 0 & 0 & \omega^2 & 0 & 0 & 0 & 0 & 0 & 0 \\
1 & 0 & 0 & 0 & 0 & 0 & 0 & 0 & 0 & 0 & 0 & 0 \\
0 & 1 & 0 & 0 & 0 & 0 & 0 & 0 & 0 & 0 & 0 & 0 \\
0 & 0 & \omega^2 & 0 & 0 & 0 & 0 & 0 & 0 & 0 & 0 & 0 \\
0 & 0 & 0 & 0 & 0 & 0 & 1 & 0 & 0 & 0 & 0 & 0 \\
0 & 0 & 0 & 0 & 0 & 0 & 0 & 1 & 0 & 0 & 0 & 0 \\
0 & 0 & 0 & 0 & 0 & 0 & 0 & 0 & \omega^2 & 0 & 0 & 0 \\
\hline
0 & 0 & 0 & 0 & 0 & 0 & 0 & 0 & 0 & 0 & -1 & 0 \\
0 & 0 & 0 & 0 & 0 & 0 & 0 & 0 & 0 & -1 & 0 & 0 \\
0 & 0 & 0 & 0 & 0 & 0 & 0 & 0 & 0 & 0 & 0 & \omega^2
\end{array}\right).$$
Thus $\ker \tilde \rho_6$ is a characteristic subgroup of $F_2$ by Proposition \ref{prop:characteristic kernel}. Observe that $\tilde \rho_6(a^6)=I$, and thus $P_6 \subset \ker \tilde \rho_6$.

The top left $9 \times 9$ block is isomorphic to the representation $\rho_6=\rho_2 \otimes \rho_3$,
where these representations were taken from \S \ref{sect:finite image}. Thus, the image $\rho_6(F_2)$ is isomorphic to $C_2 \times C_2 \times H(\Z/3\Z)$, and therefore this map induces the surjective map $F_2/\ker \tilde \rho_6 \to C_2 \times C_2 \times H(\Z/3\Z)$ in the short exact sequence. The kernel of this map is isomorphic to the image $\tilde \rho_6(\ker \rho_6)$. Thus we get our exact sequence as described but with $\Z^{18}$ replaced by $\tilde \rho_6(\ker \rho_6)$. The group $\tilde \rho_6(\ker \rho_6)$ is an abelian group because matrices in $\tilde \rho_6(\ker \rho_6)$ have a $2 \times 2$ block form with copies of the identity matrix along the diagonal and a zero matrix in the lower left block. In this subgroup, multiplication is the same as addition in the top right block, and thus $\tilde \rho_6(\ker \rho_6)$ is naturally a subgroup of $\Z[\omega]^{27}$ since $\tilde \rho_6$ takes values in $\GL\big(12,\Z[\omega]\big)$. This shows that $\tilde \rho_6(\ker \rho_6)$ is a finite rank free abelian group. To figure out this rank, we observe using \cite{GAP4} that 
$C_2 \times C_2 \times H(\Z/3\Z)$
can be written as a quotient of the rank two free group $\left< a, b \right>$ by 
$$\ker \rho_6=\big\langle \big\langle a^6, b^6, [a,b]^3, \big[a,[a,b]\big], \big[b,[a,b]\big]\big\rangle\big\rangle.$$
We already know that $a^6, b^6 \in \ker \tilde \rho_6$ and can check that $[a,b]^3 \in \ker \tilde \rho_6$. Thus the abelian image $\tilde \rho_6(\ker \rho_6)$
is generated by elements of the form 
\begin{equation}
    \label{eq:generators 6}
\tilde \rho_6\big(g  \big[a,[a,b]\big] g^{-1}\big) \quad \text{and} \quad \tilde \rho_6\big(g  \big[b,[a,b]\big] g^{-1}\big)    
\end{equation}
with $g \in F_2$. Note that if $\rho_6(g_1)=\rho_6(g_2)$, then $g_1 g_2^{-1} \in \ker \rho_6$, and since $\tilde \rho_6(\ker \rho_6)$ is abelian, we have
$$\begin{array}{rcl}
\tilde \rho_6\big(g_2 \big[a,[a,b]\big] g_2^{-1}\big) & = &
\tilde \rho_6(g_1 g_2^{-1})
\tilde \rho_6(g_2 \big[a,[a,b]\big] g_2^{-1}\big) \tilde \rho_6(g_1 g_2^{-1})^{-1} \\
& = & 
\rho_6\big(g_1 \big[a,[a,b]\big] g_1^{-1}\big).
\end{array}$$
Similarly, for such $g_1$ and $g_2$, we have 
$$\tilde \rho_6\big(g_1 \big[b,[a,b]\big] g_1^{-1}\big)=\tilde \rho_6\big(g_2 \big[b,[a,b]\big] g_2^{-1}\big).$$
Since $\rho_6(F_2) \cong C_2 \times C_2 \times H(\Z/3\Z)$, to generate $\tilde \rho_6(\ker \rho_6)$ it suffices to take elements from \eqref{eq:generators 6} with one $g$ taken from each preimage $\rho_6^{-1}(M)$ where $M$ varies over elements of $\rho_6(F_2)$. Since $\rho_6(F_2) \cong C_2 \times C_2 \times H(\Z/3\Z)$, this amounts to a list of $108$ pairs of generators. This reduces the computation
of the rank of $\tilde \rho_6(\ker \rho_6)$ to a finite computation which can be done on the computer. Using SageMath \cite{sagemath}, we computed $\rank \tilde \rho_6(\ker \rho_6)=18$ so $\tilde \rho_6(\ker \rho_6) \cong \Z^{18}$.
\end{proof}

\subsection{Odd \texorpdfstring{$k \geq 5$}{k greater than or equal to 5}}

Let $k \geq 5$ be odd and define $\rho_k:F_2 \to \GL(k,\C)$ as in \eqref{eq:rho_k}.
We define $\omega=e^{\frac{2\pi i}{k}}$.

We will define an extension $\tilde \rho_k:F_2 \to \GL(k+\frac{k-3}{2},\C)$, which we found by applying the method of Theorem \ref{thm:tilde rho}. (In the next proof, we will show that $\tilde \rho_k$ arises from $\rho_k$ by this method.) Let $B$ denote the $k \times \frac{k-3}{2}$ matrix whose column vectors are given by 
\begin{equation}
\label{eq:b_j}
\vb_j=\e_{j+1}-\e_{k-j} \quad \text{for integers  $j$ with $1 \leq j \leq  \frac{k-3}{2}$},
\end{equation}
where $\e_i$ denotes the standard basis vector with a $1$ in position $i$. 
We define $\tilde \rho_k$ in block form by
\begin{equation}
    \label{eq:tilde rho k}
\tilde \rho_k(a)=\left(\begin{array}{rr}
\rho_k(a) & 0 \\
0 & I
\end{array}\right) \quad \text{and} \quad 
\tilde \rho_k(b)=\left(\begin{array}{rr}
\rho_k(b) & B \\
0 & I
\end{array}\right).
\end{equation}

\begin{thm}
\label{thm:k odd}
For each odd $k \geq 5$ the homomorphism $\tilde \rho_k$ is an oriented characteristic representation. 
The kernel $\ker \tilde \rho_k$ contains $P_k$ and is infinite index in $F_2$. Furthermore, there is a short exact sequence of groups of the form
$$1 \to \Z^{d} \to F_2/\ker \tilde \rho_k \to H(\Z/k\Z) \to 1$$
where $d=k \cdot \frac{k-3}{2} \cdot [\Q(\omega):\Q]$.
\end{thm}
\begin{proof}
Fix an odd $k \geq 5$. To simplify notation, we will use $\rho$ to denote $\rho_k$ and use $\tilde \rho$ to denote $\tilde \rho_k$ as defined in \eqref{eq:tilde rho k}. In this proof, we will show that $\tilde \rho$ is derivable from $\rho$ as described by Theorem \ref{thm:tilde rho}. The theorem then implies that $\tilde \rho$ is an oriented characteristic representation and that $P_k \subset \ker \tilde \rho$. 

Verifying that the Theorem applies requires working through \S \ref{sect:improving}. We will begin by setting up notation and applying some results from \S \ref{sect:improving} to our setting. We will then define a subspace $\CI_L$ of the space $\CA$ of affable representations. 
We will check that $\CI_L$ is contained in $\CA_k$
and that its quotient in $\CA_k/\sim$ is $N^\sim_{M,\psi}$-invariant. Then we will observe that $\tilde \rho$ as defined above coincides with the definition in \eqref{eq:tilde rho} used in Theorem \ref{thm:tilde rho} with an appropriate choice of basis. Finally we must check that the group $\tilde \rho(\ker \rho)$ is isomorphic to $\Z^d$ with $d$ as in the statement of the theorem.

Recall that $\eval_a \times \eval_b$ gives an isomorphism $\CA \to \C^{k} \times \C^k$; see Proposition \ref{prop:coordinates}. We'll find it useful to use coordinates provided by the inverse map
$$R=(\eval_a \times \eval_b)^{-1}: \C^k \times \C^k \to \CA.$$
The image of $(\va,\vb)$ is defined as in \eqref{eq:a,b}.

The subgroup $\C^k \subset \Aff(k)$ acts on $\CA$ by conjugation. For $\vv \in \C^k$ we used $\Conj_\vv$ to denote this action, and wrote $\hat \rho_1 \sim \hat \rho_2$ if there is a $\vv$ such that $\Conj_\vv(\hat \rho_1)=\hat \rho_2$. The space $\CA/\sim$ is a vector space.
By Proposition \ref{prop:conjugation}, we know that $\Conj_\vv$ acts by translation on $\CA$ and this translation vector depends linearly on $\vv$. Thus the natural map $C:\CA \to \CA/\sim$ is linear and the kernel $\CT=\ker C$ is the collection of translation vectors. By applying the formula in Proposition \ref{prop:conjugation} to the standard basis vectors $\e_1, \ldots, \e_k \in \C^k$ and our particular $\rho$, we see
\begin{equation}
\label{eq:CT}
\CT=\mathrm{span}_\C \Big(\big\{R(\0,\e_1-\e_k)\big\}\cup \big\{R\big((1-\omega^{j-1})\e_j,\e_j-\e_{j-1}\big):~2 \leq j \leq k\big\}\Big).
\end{equation}
In particular, for each $\hat \rho \in \CA$ there is a unique $\vv \in \C^k$ satisfying 
\begin{equation}
\label{eq:conjugacy canonical form}
\Conj_\vv(\hat \rho)=R(c_1 \e_1, c_2 \e_2+c_3 \e_3+ \ldots+c_k \e_k) \quad \text{for some choice of $c_1, \ldots, c_k \in \C$.}
\end{equation}
This gives a standard representative for each conjugacy class. Let $\CS \subset \CA$ denote 
those representations which can be written in the form on the right side of \eqref{eq:conjugacy canonical form}. Then $\CS$ is a section for $C$ in the sense that the restriction $C|_{\CS}:\CS \to \CA/\sim$ is an isomorphism of vector spaces, and there is a linear map $P:\CA \to \CS$ with kernel $\CT$ 
(defined by $S=C|_{\CS}^{-1} \circ C$) which stabilizes points in $\CS$. That is, $P$ is the projection to $\CS$ with leaves parallel to $\CT.$

Now consider the subspace $\CA_k \subset \CA$ consisting of those $\hat \rho \in \CA$ such that $P_k \subset \ker \hat \rho$, as originally defined in \eqref{eq:k-affable}. Define
\begin{equation}
\label{eq:CI_L}
\CI_L=\mathrm{span}_\C \left\{R(\0,\vb_j):~ 1 \leq j \leq \frac{k-3}{2}\right\},
\end{equation}
where the vectors $\vb_j$ are defined as in \eqref{eq:b_j}.
Note $\CI_L \subset \CS$. Define $\CI=C(\CI_L)$. Then $\CI=\CI_L/\sim$ is a subspace of $\CA/\sim$.
We make the following claims:
\begin{enumerate}[label=Claim \arabic*., ref=Claim \arabic*, leftmargin=7em]
\item $\CI$ is $N_{M,\psi}^\sim$-invariant for all $(M,\psi) \in \Delta$.
\label{claim1}
\item $\CI \subset \CA_k$.
\label{claim2}
\end{enumerate}
This will verify the hypotheses of Theorem \ref{thm:tilde rho} providing a new oriented characteristic representation $\tilde \rho$ with $P_k \subset \ker \tilde \rho$. Let 
\begin{equation}
\label{eq:basis}    
\hat \rho_j=R(\0,\vb_j)
\quad \text{for} \quad 
j \in \{1, \ldots, \frac{k-3}{2}\}.
\end{equation}
We obtain the matrix representation for $\tilde \rho$ using \eqref{eq:tilde rho}.

We will see that $\CI$ has algebraic significance which explains the invariance in \ref{claim1}. 

First consider the kernel of the natural projection $\pi_2: \Delta \to \Aut(F_2)$. This subgroup consists of those pairs $(M,\text{id})$ such that $M$ commutes with every $\rho(g)$. Since $\rho$ is irreducible (Corollary \ref{cor:irreducible}) Schur's Lemma implies that only the center of $\GL(k,\C)$ commutes with all of $\rho(F_2)$. Thus, 
\begin{equation}
\label{eq:ker pi_2}
\ker \pi_2 = \{(zI, \text{id}):~ z \in \C \smallsetminus \{0\}\}.
\end{equation}

Let $\pi_2'$ denote the natural map $\Delta \to \Out(F_2)$. If $(M,\psi) \in \ker \pi_2'$, then there is an $h$ so that $\psi(g)=hgh^{-1}$ for all $g \in F_2$. Fix this $\psi$ for this discussion. Observe that one $M$ which satisfies $(M,\psi) \in \Delta$ is $\rho(h)$ (see Proposition \ref{prop:inner automorphisms act trivially}).
The other solutions differ by multiplication by an element of $\ker \pi_2$, so we have $(M,\psi) \in \Delta$ if and only if $M=z \rho(h)$ for some $z \in \C \smallsetminus \{0\}$. Then by recalling Proposition \ref{prop:inner automorphisms act trivially}, we conclude that 
for any $(M,\psi) \in \Delta$ with $\psi$ an inner automorphism, there is a $z \in \C \smallsetminus \{0\}$
such that 
\begin{equation}
\label{eq:N and inner automorphisms}
N^\sim_{M,\psi}([\hat \rho])=z [\hat \rho] \quad \text{for all $[\hat \rho] \in \CA/\sim$.}
\end{equation}

\begin{comment}
\ref{claim1} is a consequence of the group theoretic structure of $N^\sim:\Delta \to \GL(\CA/\sim)$. Namely,
that the image $N^\sim(\Delta)$ is a quotient of $\GL(2,\Z)$. To see this first observe that $N(M,\psi)$ only depends on $\psi$. This follows from the fact that $N(M,\psi)$ is trivial whenever $\psi$ is the identity automorphism. From \eqref{eq:Delta} $(M,\psi) \in \Delta$ with $\psi$ the identity if and only if $M$ commutes with each $\rho(g)$. Since $\rho$ is irreducible (Corollary \ref{cor:irreducible}) only the center of $\GL(k,\C)$ commutes with all of $\rho(F_2)$. But by definition of $N$ in Lemma \ref{lem:N},
we see that $N(M,\psi)=N(zM,\psi)$ for all $(M,\psi) \in \Delta$ and all $z \in \C$. 
We have shown that $N(M,\psi)$ only depends on $\psi$, so it follows that $N^\sim(M,\psi)$ only depends on $\psi$. By Proposition \ref{prop:inner automorphisms act trivially}, it then follows that $N^\sim(M,\psi)$ is trivial whenever $\psi$ is an inner automorphism of $F_2$. Thus the image $N^\sim(M,\psi)$ only depends on the outer automorphism class of $\psi$, i.e., $N^\sim$ factors through $\Out(F_2) \cong \GL(2,\Z)$ as claimed.
\end{comment}

Let $\psi_1 \in \Aut(F_2)$ be as in \eqref{eq:generators for Out}. Let $M_1 \in \GL(k,\C)$ be the matrix in Proposition \ref{prop:rho odd}. Then we have $(M_1,\psi_1) \in \Delta$. We claim that every element of $N^\sim_\Delta$ preserves the eigenspaces of $N^\sim(M_1,\psi_1)^2$. (In this proof, we will use $N^\sim(M,\psi)$ to denote $N^\sim_{M,\psi}$
and $N(M,\psi)$ to denote $N_{M,\psi}$ to avoid double subscripts.)
This has to do with the fact that 
the outer automorphism class of $\psi_1^2$ represents $-I$ in the identification of
$\Out(F_2)$ with $\GL(2,\Z)$, and thus lies in the center of $\Out(F_2)$.
To understand these eigenspaces are invariant, 
first recall that $\psi_1^4$ is the trivial automorphism of $F_2$. Thus, by \eqref{eq:ker pi_2}, $M_1^4$ is a non-zero scalar multiple of the identity. Note that $N^\sim(M_1, \psi_1)^4$ also scales by the same amount. It follows (say by considering Jordan canonical form) that
$N^\sim(M_1, \psi_1)^2$ is diagonalizable and has eigenvalues in the set $\{\pm z_1\}$ for some $z_1 \in \C \smallsetminus \{0\}$. 
Let $d_+$ and $d_-$ denote the dimensions of eigenspaces with eigenvalue $z_1$ and $-z_1$, respectively.
Since $\dim (\CA/\sim)=k$ (which follows from \eqref{eq:conjugacy canonical form}), we have $d_+ + d_-=k$.
Since $k$ is odd it follows that $d_+ \neq d_-$.
To verify that these eigenspaces are preserved, pick any $(M,\psi) \in \Delta$. Since the image of $\psi_1^2$ in $\Out(F_2)$ is central, we know that the commutator $[\psi^{-1}, \psi_1^{-2}]$ is an inner automorphism. Thus, by \eqref{eq:N and inner automorphisms}, $N^\sim([\psi^{-1}, \psi_1^{-2}],[M^{-1},M_1^{-2}])$ scales elements of $\CA/\sim$ by some $z \in \C \smallsetminus \{0\}$, and by simplifying we get
\begin{equation}
    \label{eq:central}
N^\sim(M,\psi) \circ N^\sim(M_1,\psi_1)^2 \circ N^\sim(M,\psi)^{-1} = z N^\sim(M_1^2,\psi_1^2).
\end{equation}
Observe that the left side above is conjugate to $N^\sim(M_1,\psi_1)^2$ and so has eigenspaces of dimension $d_\pm$ with corresponding eigenvalues of $\pm z_1$. The right hand side however has eigenvalues of dimension $d_\pm$ with corresponding eigenvalues of $\pm z z_1$. It follows that $z=1$, and then \eqref{eq:central} gives centrality of $N^\sim(M_1,\psi_1)^2$ in $N^\sim_\Delta$ and this centrality implies that the eigenspaces of $N^\sim(M_1,\psi_1)^2$ must be preserved by elements of $N^\sim_\Delta$.

We will now find a basis of eigenvectors for $N^\sim(M_1,\psi_1)^2$ to show that $\CI$ is an eigenspace. Observe that $\psi_1^2:F_2 \to F_2$ is the automorphism satisfying 
$$\psi_1^2(a)=a^{-1} \quad \text{and} \quad \psi_1^2(b)=b^{-1}$$
and thus $\psi_1^2$ is an involution. Also the entries of $M_1^2$ are given by 
$$
\arraycolsep=1.4pt
\begin{array}{rcl}
(M_1^2)_{i,j} & = &  \sum_{\ell=1}^k \omega^{(i-1)(\ell-1)} \omega^{(\ell-1)(j-1)} \\
& = & \sum_{\ell=1}^k \omega^{(j+i-2)(\ell-1)} = \begin{cases}
k & \text{if $j+i-2 \equiv 0 \pmod{k}$} \\
0 & \text{otherwise.}
\end{cases}
\end{array}
$$
In particular,
\begin{equation}
\label{eq:M12 by e}
M_1^2 \e_j = k \e_i \quad \text{where $i$ is such that $j+i-2 \equiv 0 \pmod{k}$}.
\end{equation}

Recalling the notation in the paragraph including \eqref{eq:CT} and \eqref{eq:conjugacy canonical form}, we define the map $N_1^2:\CS \to \CS$ by
\begin{equation}
    \label{eq:N12}
N_1^2 = C|_{\CS}^{-1} \circ N^\sim(M_1,\psi_1)^2 \circ C|_{\CS} =
P \circ N(M_1,\psi_1)^2|_{\CS}.
\end{equation}
Equality of these two expressions follows from the facts that 
$$C \circ N(M_1,\psi_1)^2=N^\sim(M_1,\psi_1)^2 \circ C$$ 
(i.e., that $N^\sim(M_1,\psi_1)^2$ is the action on $\CA /\sim$ induced by $N(M_1,\psi_1)^2$) and that $P=C|_{\CS}^{-1} \circ C$ as noted in the paragraph cited above. We will evaluate $N_1^2$ using the rightmost identity in \eqref{eq:N12}, by applying $N(M_1,\psi_1)^2$ followed by the projection $P:\CA \to \CS$ which has fibers parallel to $\CT$. We will show that a list of eigenvalues and eigenvectors of $N_1^2$ is given by:
\begin{enumerate}
\item[(a)] The vectors $\hat \rho_j=R(\0,\vb_j)$ are eigenvectors with eigenvalue $k$ for $j \in \{1, \ldots, \frac{k-3}{2}\}$. 
\item[(b)] The vectors $R(\0,\e_{j+1}+\e_{k-j})$ are eigenvectors with eigenvalue $-k$ for $j \in \{1, \ldots, \frac{k-3}{2}\}$. 
\item[(c)] The vectors $R(\e_1,\0)$, $R(\0,\e_{\frac{k+1}{2}})$ and $R(\0,\e_k)$ are eigenvectors with eigenvalue $-k$. 
\end{enumerate}
The reader will observe that the vectors listed above span $\CS$ and the eigenspace formed by the span of the eigenvectors in case (a) coincides with $\CI_L$. Thus by proving these statements we will have verified \ref{claim1}. 

Before proving (a)-(c) we need to understand the action of $N(M_1,\psi_1)^2$. Let $(\va,\vb) \in \C^k\times \C^k$ and $\hat \rho=R(\va,\vb)$. We have by definition of $N$:
$$N(M_1,\psi_1)^2(\hat \rho)(a) = M_1^2 \cdot \big(-\rho(a^{-1}) \va,\rho(a^{-1})\big) \cdot M_1^{-2}=
\big(-M_1^2 \rho(a^{-1}) \va,M_1^2 \rho(a^{-1})M_1^{-2}\big).$$
Since $(M_1^2, \psi_1^2) \in \Delta$, we know that $M_1^2 \rho(a^{-1}) M_1^{-2}=\rho(a)$ and thus
$$N(M_1,\psi_1)^2(\hat \rho)(a) = 
\big(-\rho(a) M_1^2 \va, \rho(a)\big).$$
Similarly, we have $N(M_1,\psi_1)^2(\hat \rho)(b) = 
\big(-\rho(b) M_1^2 \vb, \rho(a)\big).$ Putting these two together we see that
$$N(M_1,\psi_1)^2 \circ R(\va, \vb) = R \big(-\rho(a) M_1^2 \va, -\rho(b) M_1^2 \vb).$$
We specialize this using our understanding of $M_1^2$ and $\rho(a)$ and $\rho(b)$ into some useful special cases. We have
\begin{equation}
\label{eq:NM12 image of e1}
\arraycolsep=1.4pt
\begin{array}{rcl}
N(M_1,\psi_1)^2 \circ R(\e_1,\0) & = & R \big(-\rho(a) M_1^2 \e_1, \0) = R \big(-k \rho(a) \e_1, \0) \\
& = & R \big(-k \e_1, \0\big).
\end{array}
\end{equation}
For $j>1$, we have
\begin{equation}
\label{eq:NM12 image of ei}
\arraycolsep=1.4pt
\begin{array}{rcl}
N(M_1,\psi_1)^2 \circ R(\0,\e_j) & = & R \big(\0,-\rho(b) M_1^2 \e_j) = R \big(\0, -k \rho(b) \e_{k+2-j}) \\
& = &  R \big(\0, -k \e_{k+1-j}).
\end{array}
\end{equation}

Now we will check (a)-(c). First consider (c). 
That $N_1^2 \circ R(\e_1,\0)=-k R(\e_1,\0)$ follows from \eqref{eq:NM12 image of e1}. Similarly, that $R(\0,\e_{\frac{k+1}{2}})$ has eigenvalue $-k$ follows from \eqref{eq:NM12 image of ei} with $j=\frac{k+1}{2}$. Again by \eqref{eq:NM12 image of ei}, $N_1^2 \circ R(\0,\e_k)$ is the projection of $R \big(\0, -k \e_{1})$ to $\CS$ along $\CT$. Since $R(\0,\e_1-\e_k) \in \CT$, we have 
$N_1^2 \circ R(\0,\e_k) = R(\0,-k\e_k)$ finishing the proof of (c). Now consider (a). Recall that $\vb_j=\e_{j+1}-\e_{k-j}$ for $j=1, \ldots, \frac{k-3}{2}$, and using \eqref{eq:NM12 image of ei}, we observe
$$N(M_1,\psi_1)^2 \circ R(\0,\vb_j) = 
R(\0, -k\e_{k-j}+k\e_{j+1}) = k R(\0, \vb_j),$$
which verifies (a). Finally consider (b). For $j \in \{1, \ldots, \frac{k-3}{2}\}$, we have
$$N(M_1,\psi_1)^2 \circ R(\0,\e_{j+1} + \e_{k-j}) = 
R(\0, -k\e_{k-j}-k\e_{j+1})=-k R(\0,\e_{j+1} + \e_{k-j}).$$
This completes the proof of \ref{claim1}.

Finally we need to prove \ref{claim2} that $\CI \subset \CA_k/\sim$. From the above we know that $\CI$ is $N^\sim$-invariant, so it suffices to prove that $\hat \rho(a^k)=1$ for each $\hat \rho \in \CI_L$. 
We clearly have this since each $\hat \rho=R(\0,\vv)$ for some $\vv \in \C^k$, see \eqref{eq:CI_L}. 
This proves \ref{claim2}.

Since we have proven the claims, we obtain the representation $\tilde \rho$ as discussed surrounding \eqref{eq:basis}.
It remains to show that $\tilde \rho(\ker \rho) \cong \Z^d$ with $d=k \cdot \frac{k-3}{2} \cdot n$ with $n=[\Q(\omega):\Q]$ as stated in the theorem. Observe that $\tilde \rho$ is a representation from $F_2$ into $\GL\big(k+\frac{k-3}{2}, \Z[\omega]\big)$.
For $g \in \ker \rho$, the matrix $\tilde \rho(g)$ has a block form as in \eqref{eq:tilde rho}, with the identity appearing in the diagonal blocks, zero in the bottom left, and a $k \times \frac{k-3}{2}$ matrix $Q(g)$ in the top right.
We will show that the rank $d$ is as large as possible: as large as the rank $k \cdot \frac{k-3}{2} \cdot n$ of the additive group of $k \times \frac{k-3}{2}$ matrices with entries in $\Z[\omega]$. We claim that it suffices to find a $g \in \ker \rho$ such that the top right block $Q(g)$ has linearly independent columns. We will prove this suffices, and then give such a $g$ below. Observe that given any $h \in F_2$, we have $hgh^{-1} \in \ker \rho$ and a computation shows that
\begin{comment}
$$\tilde \rho(hgh^{-1})=\left(\begin{matrix}
\rho(h) & Q(h) \\
0 & I
\end{matrix}\right)
\left(\begin{matrix}
I & Q(g) \\
0 & I
\end{matrix}\right)
\left(\begin{matrix}
\rho(h)^{-1} & -\rho(h)^{-1} Q(h) \\
0 & I
\end{matrix}\right)
=
\left(\begin{matrix}
I & \rho(h) Q(g) \\
0 & I
\end{matrix}\right)
$$
so 
\end{comment}
$Q(hgh^{-1})=\rho(h) Q(g)$. It follows that
$$Q(\ker \rho) \supset \Lambda Q(g),$$
where $\Lambda$ is the additive group of matrices generated by $\rho(F_2)$. Proposition \ref{prop:matrix group} guarantees that the additive group $\mathbf M$ of $k \times k$ matrices with entries in $\Z[\omega]$ contains $\Lambda$ as a finite index subgroup. Thus 
we can find matrices $M_1, \ldots, M_{k^2 n} \in \Lambda$
which generate the space of $k \times k$ matrices with entries in $Q(\omega)$ as a $\Q$-vector space. Define the map $\Phi$ to send a $k \times k$ matrix $M$ with entries in $\Q(\omega)$
to the product $M Q(g)$. Then $\Phi$ is $\Q(\omega)$-linear so we have
$$\dim_{\Q(\omega)} (\ker \Phi) + \dim_{\Q(\omega)} (\img \Phi) = \dim_{\Q(\omega)} \mathbf M = k^2.$$
The kernel of $\Phi$ consists of those $M \in \mathbf M$ so that the rows of $M$ are perpendicular to each column of $Q(g)$. Since the columns of $Q(g)$ are linearly independent, the rows of matrices in $\ker \Phi$ can be taken from a $\Q(\omega)$-linear subspace of codimension $\frac{k-3}{2}$. We conclude $\dim_{\Q(\omega)} (\ker \Phi)=k (k-\frac{k-3}{2})$ and it follows that 
$$\dim_{\Q(\omega)} (\img \Phi) = k \cdot \frac{k-3}{2} \quad \text{and so} \quad \dim_{\Q} (\img \Phi)=k \cdot \frac{k-3}{2} \cdot n.$$ 
The images $\Phi(M_1), \ldots,  \Phi(M_{k^2 n})$ of our $\Q$-basis of matrices span the image of $\Phi$ as a $\Q$-vector space, so we can find $k \cdot \frac{k-3}{2} \cdot n$ such images which are linearly independent over $\Q$. These images freely generate a free abelian group, which lies in $\Lambda Q(g)$ and therefore also in $Q(\ker \rho)$. We conclude that the rank of $Q(\ker \rho)$ is at least $k \cdot \frac{k-3}{2} \cdot n$ as desired.

We carry out this calculation for $g=
\big[a^{-1},[a,b^{-1}]\big]=  
aba^{-1}b^{-1} a^{-1} b a b^{-1} \in \ker(\rho)$.
The columns of $Q(g)$ are given by $\pi_1 \circ \hat \rho_j(g)$ for $j \in \{1, \ldots, \frac{k-3}{2}\}$. It may be computed that
$$\hat \rho_j(aba^{-1}b^{-1})=\big((\omega^j-\omega^{-1})\e_{j+1}+(\omega^{-1}-\omega^{-j-1})\e_{k-j},\omega^{-1} I\big),$$
$$\hat \rho_j(a^{-1} b a b^{-1})=\big((\omega^{-j}-\omega^{1})\e_{j+1}+(\omega^{1}-\omega^{j+1})\e_{k-j},\omega I\big),$$
$$\hat \rho_j(g)=
\big(
(\omega^j-1)(1-\omega^{-j-1})(\e_{j+1}-\e_{k-j}),I\big).$$
(This calculation was done by hand and independently verified using \cite{sagemath} for several values of $k$.)
Observe that the coefficient $(\omega^j-1)(1-\omega^{-j-1})$ is never zero for the range of $j$ under consideration. Also the vectors are linearly independent since the positions of non-zero entries never coincide. Thus
the above argument gives us the rank we claimed.
\end{proof}

\subsection{Case $k=4$}
\label{sect:reprentations 4}
We define $\rho_4:F_2 \to \GL(2,\C)$ by 
$$\rho_4(a)=\left(\begin{array}{rr}
i & 0 \\
0 & -i
\end{array}\right) \quad \text{and} \quad
\rho_4(b)=\left(\begin{array}{rr}
0 & 1 \\
-1 & 0
\end{array}\right).$$

\begin{prop}
\label{prop:rho 4}
The image $\rho_4(F_2)$ is isomorphic to the quaternion group $Q$ of order eight.
The representation $\rho_4$ is oriented characteristic and $P_4 \subset \ker \rho_4$. The additive subgroup of $2 \times 2$ matrices generated by the image of $\rho_4$ consists of those matrices of the form
$$M_{x,y} = \left(\begin{array}{rr}
x & -\overline{y} \\
y & \overline{x}
\end{array}\right)
\quad \text{with $x,y \in \Z[i]$}.
$$
\end{prop}
\begin{proof}
Let ${\mathbf M}=\{M_{x,y}:~x,y \in \C\}$ and observe that
${\mathbf M}$ is closed under multiplication. Thus $\SL(2,\C) \cap {\mathbf M}$ is a multiplicative group containing $\rho_4(a)$ and $\rho_4(b)$. Furthermore,
$\det M_{x,y}=|x|^2+|y|^2$, so there are exactly eight matrices in $\SL(2,\C) \cap {\mathbf M}$.
%, namely those where
%\begin{equation}
%    \label{eq:det 1 solutions}
%    (x,y) \in \{(\pm 1, 0\), (\pm i, 0), (0, \pm 1), (0, \pm i)\}.
%\end{equation}
Observe by inspection that $\SL(2,\C) \cap {\mathbf M}$ is isomorphic to $Q$ and that $\rho_4(a)$ and $\rho_4(b)$ generate. Also observe that the matrices $M_{1,0}, M_{i,0}, M_{0,1}, M_{0,i} \in \rho_4(F_2)$ generate $\mathbf M$ as an additive group. To see $\rho_4$ is oriented characteristic, observe it satisfies Proposition \ref{prop:characteristic criterion} with the choice of matrices
$$
M_1=\left(\begin{array}{rr}
1 & i \\
i & 1
\end{array}\right), \quad
M_2=\left(\begin{array}{rr}
i & 0 \\
0 & 1
\end{array}\right)
\quad \text{and} \quad 
M_-=I.$$
Also we have $P_4 \subset \ker \rho_4$ since the kernel is characteristic and $\rho_4(a)^4=I$.
\end{proof}

Let $\tilde \rho_4:F_2 \to \GL(4,\C)$ be defined by 
$$\tilde \rho_4(a)=\left(\begin{array}{rr|rr}
i & 0 & 0 & 0 \\
0 & -i & 0 & 0 \\
\hline
0 & 0 & 1 & 0 \\
0 & 0 & 0 & 1
\end{array}\right)
\quad \text{and} \quad
\tilde \rho_4(b)=\left(\begin{array}{rr|rr}
0 & 1 & 1 & 0 \\
-1 & 0 & 0 & 1 \\
\hline
0 & 0 & 1 & 0 \\
0 & 0 & 0 & 1
\end{array}\right).$$
This representation was produced by following the argument of Theorem \ref{thm:tilde rho} with $\CI=\CA_{4}/\sim$, though we will not prove this. We do show:

\begin{prop}
\label{prop:trho4}
The homomorphism $\tilde \rho_4$ is an oriented characteristic representation. The kernel of $\tilde \rho_4$ contains $P_4$ and is infinite index in $F_2$. 
We have
$$\tilde \rho_4(\ker \rho_4)=\left\{\left(\begin{array}{rr|rr}
1 & 0 & z & -\overline{w} \\
0 & 1 & w & \overline{z} \\
\hline
0 & 0 & 1 & 0 \\
0 & 0 & 0 & 1
\end{array}\right):~(w,z) \in \Lambda\right\}$$
where $\Lambda$ is the kernel of the map 
$$\Z[i]^2 \to \Z/2\Z \quad \text{given by} \quad 
(a+ib,c+id) \mapsto a+b+c+d \pmod{2}.$$
%\subset \Z[i]^2$ is the index two subgroup generated by $(-1,1)$, $(-i,-i)$, $(-1,-1)$ and $(0,i+1)$. 
Thus there is a short exact sequence of groups of the form
$$1 \to \Z^{4} \to F_2/\ker \tilde \rho_4 \to Q \to 1.$$
\end{prop}
\begin{proof}
To see $\tilde \rho$ is oriented characteristic apply Proposition \ref{prop:characteristic criterion} with 
$M_-=I$, 
$$M_1=\left(\begin{array}{rr|rr}
2 & 2 i & i - 1 & -i - 1 \\
2 i & 2 & -i + 1 & -i - 1 \\
\hline
0 & 0 & 2 i - 2 & 0 \\
0 & 0 & 0 & -2 i - 2
\end{array}\right)
\quad \text{and} \quad
M_2=\left(\begin{array}{rr|rr}
i & 0 & 0 & 0 \\
0 & 1 & 0 & 0 \\
\hline
0 & 0 & 1 & 0 \\
0 & 0 & 0 & i
\end{array}\right).$$
Then to see that $P_4 \subset \ker \tilde \rho_4$, it suffices to observe that $\tilde \rho_4(a^4)=I$.

Define $\gamma:\Z[i]^2 \to \GL(4,\C)$ by 
\begin{equation}
\label{eq:gamma}
\gamma(z,w)=\left(\begin{array}{rr|rr}
1 & 0 & z & -\overline{w} \\
0 & 1 & w & \overline{z} \\
\hline
0 & 0 & 1 & 0 \\
0 & 0 & 0 & 1
\end{array}\right).
\end{equation}
The Proposition claims that $\tilde \rho_4(\ker \rho_4)=\gamma(\Lambda)$.
Recall that the quaternion group has a presentation of the form 
$$Q=\langle a,b ~|~ a^4=b^4=a^2 b^2=ab^{-1}ab=1\rangle.$$
Since $\tilde \rho_4(a^4)=\tilde \rho_4(b^4)=I$, it follows that $\tilde \rho_4(\ker \rho_4)$ is generated by images under $\tilde \rho_4$ of conjugates of $a^2 b^2$ and $ab^{-1}ab$. We compute 
$$\tilde \rho_4(a^2 b^2)=\gamma(-1,1)
\quad \text{and} \quad 
\tilde \rho_4(a b^{-1} a b)=\gamma(0,i+1).
$$
Now we will consider $\tilde \rho_4(g a^2 b^2 g^{-1})$ for $g \in F_2$. 
Let $P$ be the top right $2 \times 2$ submatrix of $\tilde \rho_4(a^2 b^2)$ above. 
Conjugates $\tilde \rho_4(g a^2 b^2 g^{-1})$ have top right submatrix given by $\rho_4(g) \cdot P$. 
Thus $\tilde \rho_4(\ker \rho_4)$ contains all the matrices $M_{x,y} P$ where $M_{x,y}$ is in the additive group generated by $\rho_4(g)$ which was described by Proposition \ref{prop:rho 4} in terms of a vector $(x,y) \in \Z[i]^2$. We have
\begin{equation}
\label{eq:conjugation action}
M_{x,y} P = \left(\begin{array}{rr}
x & -\overline{y} \\
y & \overline{x}
\end{array}\right)\cdot
\left(\begin{array}{rr}
-1 & -1 \\
1 & -1
\end{array}\right)=
\left(\begin{array}{rr}
-x-\overline{y} & -x+\overline{y} \\
-y+\overline{x} & -y-\overline{x}
\end{array}\right).\end{equation}
Varying $(x,y)$ over $\{(1,0),(i,0),(0,1),(0,i)\}$ gives generators for the normal subgroup of $\tilde \rho_4(F_2)$
$$N_1=\langle \tilde \rho_4(g a^2 b^2 g^{-1}) ~|~ g \in F_2 \rangle.$$
Namely we see that
$$N_1=\gamma(\Lambda_1) \quad \text{where} \quad \Lambda_1=\langle(-1,1),(-i,-i),(-1,-1),(i,-i)\rangle \subset \Z[i]^2.$$
A similar calculation shows that the normal subgroup
$$N_2=\langle \tilde \rho_4(g a b^{-1} a b g^{-1}) ~|~ g \in F_2 \rangle \quad \text{is given by}$$
$$N_2=\gamma(\Lambda_2) \quad \text{where} \quad \Lambda_1=\langle(0,i+1),(0,1-i),(-1-i,0),(-1+i,0)\rangle \subset \Z[i]^2.$$
A simple calculation shows that 
$$\langle \Lambda_1,\Lambda_2\rangle=\langle(-1,1),(-i,-i),(-1,-1),(0,i+1)\rangle$$
which is a subgroup of $\Z[i]^2$ with index two. Observe that $\Lambda=\langle \Lambda_1,\Lambda_2\rangle$
and from the discussion above we have $\tilde \rho_4(\ker \rho_4)=\gamma(\Lambda)$.

The short exact sequence follows from the fact that $\gamma(\Lambda)$ is a free abelian group of rank four.
\end{proof}

Given $\tilde \rho_4$ and $\rho_4$ as above we may consider the tensor product $\tilde \rho_4'=\overline{\rho_4} \otimes \tilde \rho_4$, which is also an oriented characteristic representation by Proposition \ref{prop:basic operation}. We have $\ker \tilde \rho'_4=\ker \tilde \rho_4$ and we can view $\tilde \rho_4'$ as a homomorphism 
to $\GL(8,\C)$.

Define the homomorphism $\tilde{\tilde\rho}_4:F_2 \to \GL(9,\C)$ such that 
\begin{equation}
\label{eq:ttrho1}
\tilde{\tilde\rho}_4(a)=\diag(1, -1, -i, -i; -1, 1, i, i; 1),
\end{equation}
\begin{equation}
\label{eq:ttrho2}
\tilde{\tilde\rho}_4(b)=\left(\begin{array}{rrrr|rrrr|r}
0 & 0 & 0 & 0 & 0 & 1 & 1 & 0 & 0 \\
0 & 0 & 0 & 0 & -1 & 0 & 0 & 1 & 0 \\
0 & 0 & 0 & 0 & 0 & 0 & 1 & 0 & 1 \\
0 & 0 & 0 & 0 & 0 & 0 & 0 & 1 & 0 \\
\hline
0 & -1 & -1 & 0 & 0 & 0 & 0 & 0 & 0 \\
1 & 0 & 0 & -1 & 0 & 0 & 0 & 0 & 0 \\
0 & 0 & -1 & 0 & 0 & 0 & 0 & 0 & 0 \\
0 & 0 & 0 & -1 & 0 & 0 & 0 & 0 & -1 \\
\hline
0 & 0 & 0 & 0 & 0 & 0 & 0 & 0 & 1
\end{array}\right).
\end{equation}
The top left $8 \times 8$ submatrices of images of $\tilde{\tilde\rho}_4$ realize
$\tilde \rho_4'$. The representation $\tilde {\tilde \rho}_4$ was found by applying the approach of Theorem \ref{thm:tilde rho} to $\tilde \rho_4'$ but we will not prove this. We have:

\begin{prop}
\label{prop:ttrho}
The homomorphism $\tilde{\tilde \rho}_4$ is an oriented characteristic representation. The kernel of $\tilde{\tilde \rho}_4$ contains $P_4$. Furthermore, there is a short exact sequence of groups of the form
$$1 \to \Z^d \to F_2/\ker {\tilde {\tilde \rho}}_4 \to F_2/\ker \tilde \rho_4 \to 1$$
where $d \geq 1$.
\end{prop}
It will follow from later work that $\ker \tilde{\tilde \rho}_4=P_4$ and that $d=1$ in the statement above. See Theorem \ref{thm:faithful}.
\begin{proof}
That $\tilde{\tilde \rho}_4$ is oriented characteristic follows from Proposition \ref{prop:characteristic criterion} with 
$$M_1=\left(\begin{array}{rrrr|rrrr|r}
2 & 2 i & i - 1 & -i - 1 & -2 i & 2 & i + 1 & i - 1 & i - 1 \\
2 i & 2 & -i + 1 & -i - 1 & 2 & -2 i & -i - 1 & i - 1 & i - 1 \\
0 & 0 & 2 i - 2 & 0 & 0 & 0 & 2 i + 2 & 0 & 2 \\
0 & 0 & 0 & -2 i - 2 & 0 & 0 & 0 & 2 i - 2 & -2 \\
\hline
-2 i & 2 & i + 1 & i - 1 & 2 & 2 i & i - 1 & -i - 1 & -i - 1 \\
2 & -2 i & -i - 1 & i - 1 & 2 i & 2 & -i + 1 & -i - 1 & i + 1 \\
0 & 0 & 2 i + 2 & 0 & 0 & 0 & 2 i - 2 & 0 & -2 \\
0 & 0 & 0 & 2 i - 2 & 0 & 0 & 0 & -2 i - 2 & -2 \\
\hline
0 & 0 & 0 & 0 & 0 & 0 & 0 & 0 & 4
\end{array}\right),$$
$$M_2=\diag(1, -i, -i, 1, i, 1, 1, i, 1)$$
and $M_{-}=I$. Again we have $P_4 \subset \ker \tilde{\tilde \rho}_4$ because $\ker \tilde{\tilde \rho}_4(a^4)=I$. 

It may be observed that the upper left $8 \times 8$ submatrix of $\tilde{\tilde \rho}(g)$ is a matrix representation of $\overline{\rho_4}(g) \otimes \tilde \rho_4(g)$. Since $\ker \tilde \rho_4 \subset \ker \rho_4$, we have that $\ker (\overline{\rho_4} \otimes \tilde \rho_4)=\ker \tilde \rho_4$.
Matrices in $\tilde{\tilde \rho}(\ker \tilde \rho_4)$ therefore have the block form 
$$\left(\begin{array}{rr}
I & \vv \\
0 & 1
\end{array}\right),$$
where $I$ is the $8 \times 8$ identity matrix and $\vv$ is an $8 \times 1$ matrix with entries in $\Z[i]$. Thus $\tilde{\tilde \rho}(\ker \tilde \rho_4)$ is isomorphic to an additive subgroup of $\Z[i]^8$.
Let $d=\rank \tilde{\tilde \rho}(\ker \tilde \rho_4)$. We compute
\begin{equation}
\label{eq:image of commutator squared}
\tilde{\tilde \rho}_4([a,b]^2)=\left(\begin{array}{rrrr|rrrr|r}
1 & 0 & 0 & 0 & 0 & 0 & 0 & 0 & 2 i \\
0 & 1 & 0 & 0 & 0 & 0 & 0 & 0 & 0 \\
0 & 0 & 1 & 0 & 0 & 0 & 0 & 0 & 0 \\
0 & 0 & 0 & 1 & 0 & 0 & 0 & 0 & 0 \\
\hline
0 & 0 & 0 & 0 & 1 & 0 & 0 & 0 & 0 \\
0 & 0 & 0 & 0 & 0 & 1 & 0 & 0 & 2 i \\
0 & 0 & 0 & 0 & 0 & 0 & 1 & 0 & 0 \\
0 & 0 & 0 & 0 & 0 & 0 & 0 & 1 & 0 \\
\hline
0 & 0 & 0 & 0 & 0 & 0 & 0 & 0 & 1
\end{array}\right).
\end{equation}
Thus $[a,b]^2$ lies in $\ker \tilde \rho_4$ and its image generates a copy of $\Z$ in $\tilde{\tilde \rho}_4(F_2)$. This shows $d \geq 1$.
Finally observe that we have the natural short exact sequence
$$1 \to \tilde{\tilde \rho}(\ker \tilde \rho_4) \to {\tilde {\tilde \rho}}_4(F_2) \to (\overline{\rho_4} \otimes \tilde \rho_4)(F_2) \to 1.$$
Here, the map ${\tilde {\tilde \rho}}_4(F_2) \to (\overline{\rho_4} \otimes \tilde \rho_4)(F_2)$ is the map that takes a matrix in ${\tilde {\tilde \rho}}_4(F_2)$ to its top left $8 \times 8$ block. We have 
${\tilde {\tilde \rho}}_4(F_2) \cong F_2/ \ker {\tilde {\tilde \rho}}_4$,
and we have
$(\overline{\rho_4} \otimes \tilde \rho_4)(F_2) \cong F_2/ \ker \tilde \rho_4$ from the discussion above. 
This yields the exact sequence in the proposition.
\end{proof}

%% file: square-tiled.tex
\section{Relation to square tiled surfaces}
\label{appendix}
A {\em translation surface} is a surface equipped with an atlas of coordinate charts to the plane such that all transition functions are restrictions of translations.  

Let $\T$ denote the $2$-torus $\R^2/\Z^2$ and $\T^\circ=\T \smallsetminus \{\0\}$ be the once punctured torus. A {\em square-tiled surface} (or {\em origami}) is a cover of $\T^\circ$ endowed with the pullback translation structure. Here we allow the cover to be finite or infinite. 
See \cite{Zorich06} for a survey discussing translation surfaces including square-tiled surfaces.

Fix a translation surface $S$. Given a vector $(u,v) \in \R^2$ the {\em straight-line flow determined by $(u,v)$} is the flow $F^t:S \to S$ given in local coordinates by
$$F^t(x,y)=(x,y)+t(u,v).$$
The straight line flow of a point will not be defined for all time if under the projection to $\T$
the flow hits the puncture at $\0$. We call such a straight-line trajectory {\em singular}.

Let $(u,v) \in \Z^2$ and assume $u$ and $v$ are relatively prime. Then the straight-line flow determined by $(u,v)$ on the torus $\T$ is periodic with all points having period one. Let $S$ be a square tiled surface. For a positive integer $k$ we say $S$ is {\em $k$-periodic} if for all relatively prime $(u,v) \in \Z^2$, every non-singular straight-line trajectory determined by $(u,v)$ is periodic with period dividing $k$.

We take $(\frac{1}{2},\frac{1}{2})$ to be the basepoint of $\T^\circ$ and say that 
a {\em square tiled surface with basepoint} is a square tiled surface $S$ with the choice of a basepoint $s$ such that the covering map to $\T^\circ$ maps $s$ to $(\frac{1}{2},\frac{1}{2})$. If $S_1$ and $S_2$ are two square tiled surfaces with basepoints $s_1$ and $s_2$ respectively and $\pi_i:S_i \to \T^\circ$ are the associated covering maps we say that  $S_1$ {\em covers} $S_2$ if there is a covering map $\pi:S_1 \to S_2$ satisfying $\pi(s_1)=\pi(s_2)$ and $\pi_2 \circ \pi=\pi_1$.

This paper originated with the following observation:
\begin{prop}
\label{prop:universal square tiled surface}
For any $k \geq 1$ there is a $k$-periodic square tiled surface with basepoint $U_k$
such that $U_k$ covers any other $k$-periodic square-tiled surface with basepoint.
\end{prop}
We call $U_k$ the {\em universal $k$-periodic square-tiled surface.}

Covering space theory associates a square tiled surface $S$ with basepoint to a subgroup $\Gamma_S$ of the fundamental group $\pi_1\big(\T^\circ,(\frac{1}{2},\frac{1}{2})\big)$. Note that this fundamental group is isomorphic to the free group $F_2$. For purposes of this appendix consider $\pi_1\big(\T^\circ,(\frac{1}{2},\frac{1}{2})\big)$ to be the same as $F_2$.
Following Herrlich we call $S$ {\em characteristic} if $\Gamma_S$ is a characteristic subgroup of $F_2$. Characteristic square-tiled surfaces $S$ are maximally symmetric: they have a deck group acting transitively on the lifts of any point of $\T^\circ$ and each element of $\GL(2,\Z)$ stabilizes $S$ (through the action of $\GL(2,\R)$ on the space of translation surfaces).

Some finite characteristic square-tiled surfaces which are $k$-periodic have attained an almost mythical status in the subject of translation surfaces serving up numerous counterexamples in the field. Especially famous are the fantastically named {\em eierlegende Wollmilchsau} discovered independently in \cite{Forni05} and \cite{HS08} and the {\em ornithorynque} first described in \cite{FM08}. These surfaces were  studied further in \cite{FMZ11} and \cite{Matheus14}. 
If this article were written more geometrically the Heisenberg origamis studied by Herrlich in
\cite{Herrlich06} would play a leading role.

Two facts combine to give a proof of Proposition \ref{prop:universal square tiled surface}:
\begin{enumerate}
\item From basic covering space theory, the square-tiled surface with basepoint $S_1$ covers the square-tiled surface $S_2$ with basepoint if and only if $\Gamma_{S_2} \subset \Gamma_{S_1}$. 
\item A conjugacy class in $F_2$ represents a homotopy class of curves containing closed geodesics on $\T^\circ$ if and only if the conjugacy class consists of primitive elements in $F_2$. This observation dates back to Jakob Nielsen's 1913 Thesis. 
\end{enumerate}
It follows that a square-tiled surface with basepoint $S$ is $k$-periodic if and only if it is covered by the square tiled surface $U_k$ defined such that $\Gamma_{U_k}=P_k$ where $P_k \subset F_2$ denotes the subgroup generated by $k$-th powers of primitive elements as in this paper.

From work in this paper we obtain an understanding of $U_1, \ldots, U_4$:
\begin{enumerate}
\item We have $U_1=\T^\circ$. 
\item The surface $U_2$ is $(\R/2\Z)^2$ punctured at the integer points.
\item The surface $U_3$ is the Heisenberg origami denoted $O_{3,3}$ in \cite{Herrlich06} jointly discovered by Herrlich, M\"oller and Weitze-Schmithuesen.
\end{enumerate}
The eierlegende Wollmilchsau mentioned above is the square-tiled surface $W$ defined such that $\Gamma_W$ is the kernel of the surjective homomorphism $F_2 \to Q$ where $Q$ is the quaternionic group. 
The surface $W$ is $4$-periodic. From our understanding in this paper of $P_4$ and in particular knowledge of the representation $\tilde{\tilde\rho}_4$ of \S \ref{sect:reprentations 4} which is faithful by Theorem \ref{thm:faithful}, we see:

\begin{thm}
The surface $U_4$ is an infinite area square tiled surface and is a torsion-free $5$-dimensional $2$-step nilpotent cover of the eierlegende Wollmilchsau.
\end{thm}

It is particularly interesting that $U_4$ is a geometrically natural example of an infinite nilpotent cover of a compact translation surface, because some methods are available to study the dynamics of the straight-line flow on such a surface; see for instance \cite{Conze09}. It is a consequence of \cite[Theorem G.3, Remark 4.1]{Hooper15} and $\GL(2,\Z)$-invariance of $U_4$ that:
\begin{cor}
\label{cor:ergodicity}
There is a dense subset $E$ of the unit circle in $\R^2$ with Hausdorff dimension larger than $\frac{1}{2}$ such that for any $(u,v) \in E$ the straight-line flow determined by $(u,v)$ on $U_4$ is ergodic.
\end{cor}
As a consequence of the universality of $U_4$ it follows that the straight-line flow determined by each $(u,v) \in E$ is ergodic on each $4$-periodic square tiled surface. This motivates:
\begin{ques}
\label{q:ergodicity}
Is it the straight-line flow determined by $(u,v)$ ergodic on $U_4$ whenever $\frac{u}{v} \not \in \Q$?
\end{ques}
The kernels of the representations $\tilde \rho_k$ for odd $k \geq 5$ determine characteristic $k$-periodic origamis $O_k$ which are infinite free abelian covers of the Heisenberg origamis of Herrlich. The conclusions of Corollary \ref{cor:ergodicity} then hold for the surfaces $O_k$ and we similarly wonder what the answer to 
Question \ref{q:ergodicity} would be in these cases. 

This paper shows that $P_k$ is of infinite index in $F_2$ when $k \geq 4$ and it follows that for $k \geq 4$ the surface $U_k$ is infinite. Virtual nilpotence of $F_2/P_k$ is necessary to apply the \cite[Theorem G.3]{Hooper15} so an affirmative answer to Question \ref{q:extrinsic}(b) in a case of $r=2$ and $k \geq 5$ would extend Corollary \ref{cor:ergodicity} to cover the corresponding $U_k$. Even in the absence of this, the method of \S \ref{sect:representations} can be iterated to produce other characteristic multi-step nilpotent covers of compact square tiled surfaces when applied multiple times in the cases of $k \geq 5$
as with our construction of the representation $\tilde{\tilde \rho}_4 : F_2/P_4 \to \GL(9,\C)$.